\documentclass{article}

\usepackage{amssymb}
\usepackage{polski}
\usepackage[english, polish]{babel}
\usepackage[utf8]{inputenc}
\usepackage{amsthm}
\usepackage{amsmath}
\usepackage{amsfonts}
\usepackage{enumerate}
\usepackage[alphabetic]{amsrefs}
\usepackage{array}
\usepackage{tabularx}

\newcolumntype{L}[1]{>{\raggedright\let\newline\\\arraybackslash\hspace{0pt}}m{#1}}
\newcolumntype{C}[1]{>{\centering\let\newline\\\arraybackslash\hspace{0pt}}m{#1}}
\newcolumntype{R}[1]{>{\raggedleft\let\newline\\\arraybackslash\hspace{0pt}}m{#1}}

\usepackage{mathrsfs}
\usepackage{url}
\usepackage[all]{xy}

\newtheorem{thm}{Theorem}
\newtheorem{prop}[thm]{Proposition}
\newtheorem{lemma}[thm]{Lemma}

\newtheorem{conj}[thm]{Conjecture}

\newtheorem*{thm*}{Theorem}
\newtheorem*{conj*}{Conjecture}
\newtheorem*{qn*}{Question}

\newtheorem{problem}[thm]{Problem}

\theoremstyle{definition}
\newtheorem{remark}{Remark}
\newtheorem*{defi}{Definition}

\newcolumntype{D}{ >{\centering\arraybackslash} m{0.7cm} }
\newcolumntype{I}{ >{\centering\arraybackslash} m{4cm} }
\newcolumntype{E}{ >{\centering\arraybackslash} m{2.5cm} }
\newcolumntype{F}{ >{\centering\arraybackslash} m{0.7cm} }
\newcolumntype{G}{ >{\centering\arraybackslash} m{5cm} }
\newcolumntype{J}{ >{\centering\arraybackslash} m{3.5cm} }

\title{Variants of the Selberg sieve, and almost prime $k$-tuples}
\author{Pawe\l  ~Lewulis\thanks{Supported by NCN Sonatina 3, 2019/32/C/ST1/00341.}}
\date{}

\begin{document}

\maketitle


\selectlanguage{english}

\begin{abstract} Let $k\geq 2$ and $\mathcal{P} (n) = (A_1 n + B_1 ) \cdots (A_k n + B_k)$ where all the $A_i, B_i$ are integers. Suppose that $\mathcal{P} (n)$ has no fixed prime divisors. For each choice of $k$ it is known that there exists an integer $\varrho_k$ such that $\mathcal{P} (n)$ has at most $\varrho_k$ prime factors infinitely often. We used a new weighted sieve set-up combined with a device called an $\varepsilon$-trick to improve the possible values of $\varrho_k$ for $k\geq 7$. As a by-product of our approach, we improve the conditional possible values of $\varrho_k$ for $k\geq 4$, assuming the generalized Elliott--Halberstam conjecture. 
\end{abstract}

\section{Introduction}
\subsection*{State of the art}

Let us begin with recalling the following notion. 

\begin{defi}[Admissible tuples] Fix a positive integer $k$. For each $i=1, \dots , k$ fix integers $A_i$, $B_i$, such that $A_i >0$, and let $L_i \colon \mathbf{Z}^+ \rightarrow \mathbf{Z}$ be a function given by the formula $ L_i (n) := A_i n + B_i$. For each positive integer $n$ put
\[ \mathcal{P} (n) := \prod_{i=1}^k L_i (n). \]
We call $\mathcal{H}:= \{L_1 , \dots , L_k \}$ an \textit{admissible $k$--tuple}, if for every prime $p$ there is an integer $n_p$ such that none of the $L_i (n_p)$ is a multiple of $p$.
\end{defi}

We are interested in the following problem being a vast generalization of the twin primes conjecture.

\begin{conj}[Dickson--Hardy--Littlewood]\label{DHL}
Fix a positive integer $k$. Let $\{ L_1, \dots , L_k \}$ be an admissible $k$--tuple. Then,
\begin{equation}
\liminf_{n \rightarrow \infty }  \Omega ( \mathcal{P}(n) ) = k.
\end{equation}
\end{conj}

One may reformulate the statement above into a general question about the total number of prime factors contained within $\mathcal{P} (n)$. This creates a way to `measure' of how far we are from proving Conjecture \ref{DHL}.

\begin{problem}[$DHL_{\Omega}$] 
Fix positive integers $k$ and $\varrho_k \geq k$. Let $\{ L_1, \dots , L_k \}$ be an admissible $k$--tuple. The task is to prove that
\begin{equation}\label{twierdzenie_zasadnicze}
\liminf_{n \rightarrow \infty}  \Omega ( \mathcal{P} (n)) \leq \varrho_k.
\end{equation}
\end{problem}
From this point on, if inequality (\ref{twierdzenie_zasadnicze}) is true for some precise choice of $k$ and $\varrho_k$, then we say that $DHL_\Omega [k;\varrho_k]$ holds. In the case $k=1$, the classical Dirichlet's theorem is equivalent to $DHL_\Omega [1;1]$. This is also the only instance where the optimal possible value of $\varrho_k$ is already known. For $k = 2$ we have $DHL_\Omega [2;3]$ by Chen's theorem proven in \cite{Chen}. If $k\geq 3$, then the state of the art and recent progress are described below.
\begin{center}
\centering
\text{Table A. State of the art -- obtained values $\varrho_k$ for which $DHL_\Omega [k;\varrho_k]$ holds.} \\[1.2ex]
\text{Unconditional case.}
\vspace{1.2mm}
\\
\renewcommand{\arraystretch}{1}
  \begin{tabular}{ | G || F | F | F | F | F | F | F | F |@{}m{0cm}@{}}\hline    
  $k$ &   3 &  4  &  5 & 6 & 7 & 8 & 9 & 10 &  \rule{0pt}{4ex} \\ \hline 
  Halberstam, Richert \cite{Halberstam-Richert} &  10  & 15   & 19   &  24   & 29   & 34   & 39  & 45  &   \rule{0pt}{4ex}  \\
    Porter \cite{Porter} &  8  &    &   &    &    &    &     &   &   \rule{0pt}{4ex} \\
    Diamond, Halberstam  \cite{DH} &    &  12 & 16 & 20    & 25   & 29 & 34 & 39 & \rule{0pt}{4ex} \\
    Ho, Tsang  \cite{HT} &   & & & & 24 & 28 & 33  & 38  & \rule{0pt}{4ex} \\
    Maynard \cite{3-tuples , MaynardK}   & 7 & 11 & 15 & 18 & 22 & 26 & 30 & 34 & \rule{0pt}{4ex} \\
    Lewulis \cite{Lewulis}  &   &   & 14 & & & & & & \rule{0pt}{4ex} \\
     \textbf{This work} &     &   & \textbf{}  &  \textbf{}  &  \textbf{21}  &  \textbf{25}  &  
    \textbf{29}   &  \textbf{33}  & \rule{0pt}{4ex} \\
      \hline
  \end{tabular}
\end{center}
\vspace{-1mm}
\begin{center} 
\centering
\text{The $GEH$ case.}
\vspace{1.2mm}
\\
\renewcommand{\arraystretch}{1}
  \begin{tabular}{ | G || F | F | F | F | F | F | F | F |@{}m{0cm}@{}}\hline
    $k$ &  3  &  4  &  5 & 6 & 7 & 8 & 9 & 10 &  \rule{0pt}{4ex}  \\ \hline 
        Sono \cite{Sono} &  6  & &  &   &  &     &   &   &  \rule{0pt}{4ex}  \\
    Lewulis \cite{Lewulis} &    & 10 & 13  &  17  & 20  &  24   &  28  &  32 &  \rule{0pt}{4ex}  \\
    \textbf{This work} &   &  \textbf{8}  & \textbf{11}  &  \textbf{14}  &  \textbf{17}  &  \textbf{21}  &  
    \textbf{24}   &  \textbf{27}  &  \rule{0pt}{4ex}  \\
    \hline
  \end{tabular}
\end{center} 

\subsection*{Notation}

The letter $p$ with possible indices always denotes a prime number and ${\log}$ denotes the natural logarithm. We use the notation $\mathbf{N}=\{1,2,3,\dots \}$. We also use the following definitions listed below:

\begin{itemize}
\item $\varphi (n) := \# \left( \mathbf{Z} / n \mathbf{Z} \right)^\times $ denotes Euler totient function; 
\item $\tau (n)  := \sum_{d|n} 1 $ denotes the divisor function; 
\item $\Omega (n)$ denotes the number of prime factors of $n$; 
\item $\pi (x) := \# \left\{ n \in \mathbf{N}: n \leq x, ~n \text{ is prime} \right\}$;
\item $\pi (x;q,a) := \# \left\{ n \in \mathbf{N}: n \leq x,~n \equiv a \bmod q, ~n \text{ is prime} \right\}$;
\item $\log_y x := \frac{\log x}{\log y}$ for $x,y>0$ and $y \not=1$;
\item By $(a,b)$ and $[a,b]$ we denote the greatest common divisor and the lowest common multiple, respectively;
\item For a logical formula $\phi$ we define the indicator function $\mathbf{1}_{\phi (x)}$ that equals $1$ when $\phi (x)$ is true and $0$ otherwise;
\item For a set $A$ we define the indicator function $\mathbf{1}_{A}$ that equals $1$ when the argument belongs to $A$ and $0$ otherwise;
\item By $\text{gpf}(n)$ and $\text{lpf}(n)$ we denote the greatest and the lowest prime divisor of $n$ respetively;
\item The condition $n \sim x$ means that $x<n\leq 2x$;
\item For a function $F$ being a map between some two abelian groups we define the difference operator $\partial_y F(x) := F(x+y)-F(x)$; 
\item We define an analogous operator for a function $F$ with $m$ variables, namely \\ 
$\partial_y^{(i)} F(x_1, \dots , x_m) := F(x_1, \dots, x_{i-1}, x_i + y, x_{i+1}, \dots , x_m) - F(x_1, \dots , x_m)$; 
\item  For every compactly supported function $F \colon [0,+ \infty ) \rightarrow \mathbf{R}$ we define
\[S( F) := \sup \left( \{ x \in \mathbf{R} \colon F(x) \not= 0 \} \cup \{ 0 \} \right);\]
\item We define a normalizing expression $B:=\frac{\varphi (W) \log x}{W}$ (cf. next subsection);
\item Symmetric polynomials of degree $m$ and $k$ variables $\sum_{j=1}^k t_j^m$ are denoted as $P_m$;
\item For a finitely supported arithmetic function $f \colon \mathbf{N} \rightarrow \mathbf{C}$ we define a discrepancy
\[ \Delta \left( f; a \bmod q \right) := \sum_{n\equiv a \bmod q} f(n) - \frac{1}{\varphi (q)} \sum_{(n,q)=1} f(n) \, ;  \]
\item For any $f \colon \mathbf{R} \rightarrow \mathbf{R}$ we define a function related to Selberg weights
\[ \lambda_f (n) := \sum_{d | n } \mu (d) f (\log_x d); \]
\item We also make use of the `big $O$', the `small $o$', and the `$\ll$'  notation in a standard way. 
\end{itemize}

\subsection*{The general set-up}

 Let us fix $k \in \mathbf{Z}^+$ and consider the expression
\begin{equation}\label{1.1}
\mathcal{S} := \sum_{ \substack{ n \sim x \\ n \equiv b \bmod W}} \left( \varrho_k - \Omega ( \mathcal{P} (n)) \right)  \nu (n), 
\end{equation}
where $\nu $ is some arbitrarily chosen sequence of non-negative weights. Put
  \[W := \prod_{p<D_0} p\] 
for $D_0 := \log \log \log x$ and take some integer $b$ coprime to $W$. We choose some residue class $b$ such that $\mathcal{P} (b)$ is coprime to $W$ and then, we restrict our attention to $n \equiv b \bmod W$. This way we discard all irregularities caused by very small prime numbers. Put $A:= 4\max \left\{ |A_1|, |B_1| , \dots , |A_k|,|B_k| \right\}$. Assume that $x> 10 \, 000$ and $D_0>A$. Thus, our goal is to show that
\begin{equation}\label{S=}
\mathcal{S} = \varrho_k \mathcal{S}_0 - \mathcal{S}_\Omega > 0,
\end{equation}
where
\begin{equation}\label{S012}
\begin{split}
\mathcal{S}_0 &:=  \sum_{ \substack{ n \sim x \\ n \equiv b \bmod W}}   \nu (n), \\
\mathcal{S}_\Omega &:=  \sum_{ \substack{ n \sim x \\ n \equiv b \bmod W}}   \Omega( \mathcal{P} (n))  \nu (n). \\
\end{split}
\end{equation}
The main difficulty is to calculate $\mathcal{S}_\Omega$ with sufficient accuracy. One possible method and a good source of inspiration for new tools is the following identity valid for square-free $n\leqslant x$: 
\begin{equation}\label{rozbicie omega}
\Omega (n) = \sum_{p|n} 1 = \mathbf{1}_{\text{gpf}(n) > U}  + \sum_{\substack{ p|n \\ p \leq U }}1,
\end{equation}
where $U > x^{1/2}$ (usually, $U=x^{1/2+\epsilon}$ for some small $\epsilon>0$ has been considered). For instance, one can exploit the simple inequality
 \begin{equation}\label{1..9}
\Omega ( \mathcal{P} (n) )  ~=~  \sum_{i=1}^k  \mathbf{1}_{\text{gpf}(L_i (n) ) > U}  ~+ \sum_{\substack{ p| \mathcal{P} (n) \\ p \leq U }}1
 ~\leq ~  k ~+ \sum_{\substack{ p| \mathcal{P} (n) \\ p \leq U }}1
\end{equation}
under the previous assumptions.  This reasoning leads to results that are nontrivial, but weaker than already existing in literature. However, the interesting observation about this identity is that one does not need to rely on any distributional claims about primes in arithmetic progressions in order to exploit it.
 
In \cite{MaynardK} and \cite{Lewulis} the authors applied the following identity valid for all square-free $n \sim x$: 
\begin{equation}\label{identity na sumy}
\Omega (n) =  \frac{ \log n}{\log T}  +  \sum_{p|n} \left( 1 - \frac{\log  p}{\log T} \right)  ,
\end{equation}
where $T:= x^l$ for some exponent $l \in (0,1]$. This approach combined with (\ref{rozbicie omega}) gives some flexibility, because the expression in the parentheses in (\ref{identity na sumy}) is negative for $p>T$. In such case, we can transform the task of seeking for upper bounds for $\mathcal{S}_\Omega$ into problem of establishing lower bounds. The idea was to apply the following partition of unity:
\begin{equation}\label{part unity}
  1 = \sum_{r} \mathbf{1}_{\Omega (n)=r} \geq  \sum_{r \leq H} \mathbf{1}_{\Omega (n)=r},
\end{equation}
valid for any $H>0$, and then, to calculate the contribution of $S_\Omega$ via (\ref{identity na sumy}) and (\ref{part unity}), usually for $H=3,4$, depending on specific cases.

  In this work we propose a different approach and we establish the asymptotic behaviour of $\mathcal{S}_\Omega$. Such a result is sufficient to improve the currently known values of $\varrho_k$ in the conditional case, that is when $GEH$ (cf. Section `Preparing the sieve' for definitions) is true. It also greatly simplifies the unconditional results from \cite{Lewulis} and explains Conjecture 4.2 formulated there, which turns out to be slightly incorrect. 
  
  To tackle the uncondtitional case, we need to expand the sieve support beyond the domain offered by the standard claims regarding primes in arithemetic progressions (Theorem \ref{GBV}, in particular). Hence, we incorporate a device invented in \cite{Polymath8} called an $\varepsilon$-trick. In order to do so, we have to apply (\ref{identity na sumy}). The reason for this is that the $\varepsilon$-trick is all about bounding the sieve weights from below. In the same time, we wish to apply this tool to $\mathcal{S}_\Omega$, which has to be estimated from above. As we noticed before, (\ref{identity na sumy}) enables us to partially convert upper bounds into lower bounds, at least until the prime factors are sufficiently large. On the other hand, if they are small, we do not need to rely on any distributional claim on primes in arithmetic progressions at all, so in this case we can expand the sieve support almost freely.
  
  To summarize, we propose a general set-up that is flexible enough to cover all applications appearing in this work. We have the following criterion for our main problem. 
\begin{lemma}\label{criterion}
Let $k \geq 2$ and $\varrho \geq k$ be fixed integers. Suppose that for each fixed admissible $k$--tuple $\{ L_1, \dots , L_k \}$ and each residue class $b \bmod W$ such that $(L_i (b) , W)=1$ for all $i = 1, \dots , k$, one can find a non-negative weight function $\nu \colon \mathbf{N} \rightarrow \mathbf{R}^+$ and fixed quantities $\alpha>0$, and $\beta_1, \dots  \beta_k \geq 0$, such that one has the asymptotic lower bound
\begin{equation}\label{key upper bound}
 \sum_{ \substack{ n \sim x \\ n \equiv b \bmod W }}   \nu (n) \geq \left( \alpha - o(1) \right) B^{-k} \frac{x}{W},
 \end{equation}
and the asymptotic upper bounds
\begin{align}\label{key lower bounds}
  \sum_{ \substack{ n \sim x \\ n \equiv b \bmod W \\ \mathcal{P}(n) \textup{ sq-free} }}   
  \sum_{p | L_i (n) } \left( 1 - \ell \log_x p \right)    \nu (n) &\leq (\beta_i + o(1)) B^{-k} \frac{x}{W}, \\
   \sum_{ \substack{ n \sim x \\ n \equiv b \bmod W \\ \mathcal{P}(n) \textup{ not sq-free} }} \tau ( \mathcal{P} (n))  
   \left|  \nu (n) \right| &\leq  o(1) \times B^{-k} \frac{x}{W} 
\end{align}
for all $i = 1, \dots , k$, and the key inequality
\[ \varrho > \frac{\beta_1 + \dots + \beta_k}{\alpha} + \ell k. \]
Then, $DHL_\Omega [k; \varrho] $ holds. Moreover, if one replaces inequalities (\ref{key upper bound}--\ref{key lower bounds}) with equalities, then the right-hand side of the key inequality above is constant with respect to the $\ell$ variable. 
\end{lemma}

\begin{proof}
We have
\begin{multline}\label{rozwiniecie_setup}
 \sum_{ \substack{ n \sim x \\ n \equiv b \bmod W }}\left( \varrho - \Omega (\mathcal{P} (n) \right)\nu (n)  =  
\varrho \left( \sum_{ \substack{ n \sim x \\ n \equiv b \bmod W }} \nu (n) \right) - \left(  \sum_{ \substack{ n \sim x \\ n \equiv b \bmod W \\ \mathcal{P}(n) \textup{ sq-free} }}   
\Omega (\mathcal{P} (n))  \nu (n) \right) \\ +   
 O \left(  \sum_{ \substack{ n \sim x \\ n \equiv b \bmod W \\ \mathcal{P}(n) \textup{ not sq-free} }} \tau ( \mathcal{P} (n))  
   \nu (n) \right) . 
  \end{multline}
We also observe that 
\begin{multline}\label{dalsze_rozwiniecie}
  \sum_{ \substack{ n \sim x \\ n \equiv b \bmod W \\ \mathcal{P}(n) \textup{ sq-free} }}   
\Omega (\mathcal{P} (n))  \nu (n) =
 \left( \sum_{i=1}^k \sum_{ \substack{ n \sim x \\ n \equiv b \bmod W \\ \mathcal{P}(n) \textup{ sq-free} }}   
  \sum_{p | L_i (n) } \left( 1 - \ell \log_x p \right)    \nu (n) \right) \\
+ (\ell k + o(1)) \left( \sum_{ \substack{ n \sim x \\ n \equiv b \bmod W }} \nu (n) \right) .
\end{multline}
Combining (\ref{rozwiniecie_setup}--\ref{dalsze_rozwiniecie}) with the assumptions we arrive at
\begin{equation}\label{wykonczeniowka}
 \sum_{ \substack{ n \sim x \\ n \equiv b \bmod W }}\left( \varrho - \Omega (\mathcal{P} (n) \right)\nu (n)   \geq
  \left(  ( \varrho + \ell k ) \, \alpha - \sum_{i=1}^k \beta_i - o(1) \right)B^{-k} \frac{x}{W}.
\end{equation}
Note that (\ref{wykonczeniowka}) becomes an equality, if one replaces inequalities (\ref{key upper bound}--\ref{key lower bounds}) with equalities -- in such a case the left-hand side of (\ref{wykonczeniowka}) obviously does not depend on the $\ell$ variable, so the same has to be true for the right-hand side of (\ref{wykonczeniowka}). We conclude that the left-hand side of (\ref{wykonczeniowka}) is asymptotically greater than $0$ if
\begin{equation}
 \varrho > \frac{\beta_1 + \dots + \beta_k}{\alpha} + \ell k. 
\end{equation}
\end{proof}

As mentioned in Table A, the main goal of this work is to prove the following result.
\begin{thm}[Main Theorem]\label{MAIN} $DHL_\Omega [k, \varrho_k]$ holds with the values $\varrho_k$ given in a table below
\emph{
\begin{center} 
\centering
\text{Table B.}
\vspace{1.2mm}
\\
\renewcommand{\arraystretch}{1}
  \begin{tabular}{ | I || F|  F | F | F | F | F | F | F | F |@{}m{0cm}@{}}\hline
    $k$ &  2 & 3  &  4  &  5 & 6 & 7 & 8 & 9 & 10 &  \rule{0pt}{4ex}  \\ \hline 
     Unconditionally & \textit{4} & \textit{7} &  \textit{11}  & \textit{14}  & \textit{18}  & \textbf{21}  & \textbf{25}  & \textbf{29}  & \textbf{33} &  
\rule{0pt}{4ex}  \\
    Assuming $GEH$ & \textit{3} &  \textit{6} &  \textbf{8}  & \textbf{11}  &  \textbf{14}  & \textbf{17}  &  \textbf{21}  & \textbf{24}& \textbf{27}  &
  \rule{0pt}{4ex}  \\
    \hline
  \end{tabular}
\end{center}
 \emph{(bolded text indicates the novelties in the field).}   }
\end{thm}

\subsection*{Preparing the sieve} 

In this subsection we are focused on motivating our future choice of sieve weights $\nu (n)$, so this discussion will be slightly informal. Our task is to make the sum (\ref{1.1}) greater than 0 for some fixed $\varrho_k$. That would be sufficient to prove that $DHL_\Omega [k;\varrho_k]$ holds.
Hence, the weight $\nu $ has to be sensitive to almost prime $k$-tuples. We observe that the von Mangoldt function satisfies
\begin{equation*}
\Lambda (n) = \left( \mu * \log \right)(n) = - \sum_{d|n} \mu (d) \log d,
\end{equation*}
which for square-free $n \sim x$ gives
\begin{equation}
\mathbf{1}_{n \text{ is prime}} \approx \sum_{d|n} \mu (d) \left( 1 - \log_x d \right).
\end{equation}
That motivates the following construction of the Selberg sieve:
\begin{equation}\label{Selberg_1}
\mathbf{1}_{n \text{ is prime}} \lessapprox f(0) \left( \sum_{\substack{ d|n}} \mu (d) f( \log_x d) \right)^2,
\end{equation}
where $f\colon [0,+\infty) \rightarrow \mathbf{R}$ is piecewise smooth and supported on $[0,1)$. The problem is that the Bombieri--Vinogradov theorem usually forces us to assume that $\mbox{supp}(f) \subset [0, \theta)$ for some fixed positive $\theta$. The usual choice here is $\theta$ somewhat close to $1/4$, or greater, if one assumes the Elliott--Halberstam conjecture. 

In the multidimensional setting we have
\begin{equation}\label{MSS}
\mathbf{1}_{L_1(n), \dots , L_k(n) \text{ are all primes}} \lessapprox f(0,\dots, 0) \left( \sum_{\substack{ d_1, \dots , d_k \\ \forall i ~ d_i | L_i (n) }} \left( \prod_{i=1}^k \mu (d_i) \right) f \left( \log_x d_1 , \dots , \log_x d_k \right) \right)^2
\end{equation}
for some $f \colon  [0,+\infty)^k \rightarrow \mathbf{R}$ being piecewise smooth and compactly supported. In certain cases this approach can be more efficient than (\ref{Selberg_1}), as was shown in \cite{Maynard}, where it was introduced. Dealing with multivariate summations may be tedious at times, so we would like to transform the right-hand side of $(\ref{MSS})$ a bit by replacing the function $f$ with tensor products
\begin{equation}\label{niezalezne}
f_1(\log_x d_1) \cdots f_k (\log_x d_k), 
\end{equation}
where $f_1, \dots , f_k \colon \mathbf[0,+\infty) \rightarrow \mathbf{R}$. By the Stone--Weierstrass theorem we can approximate $f$ by a linear combination of functions of such form, so essentially we lose nothing here. Our more convenient sieve weights look as follows:
\begin{equation}\label{NuSelberg}
\left( \sum_{j=1}^J c_j \prod_{i=1}^k \lambda_{f_{j,i}} ( L_i (n)) \right)^2
\end{equation}
with some real coefficients $c_j$, some smooth and compactly supported functions $f_{i,j}$. Recall that
\begin{equation*}
\lambda_f (n) := \sum_{d | n } \mu (d) f (\log_x d).
\end{equation*}
It is clear that such a weight can be decomposed into linear combination of functions of the form
\begin{equation}
n \mapsto \prod_{i=1}^k  \lambda_{F_{i}} ( L_i (n))  \lambda_{G_{i}} ( L_i (n)).
\end{equation} 
In fact, (\ref{NuSelberg}) is exactly our choice in Section \ref{proof_theo}.

\subsection*{Distributional claims concerning primes}

In this work we refer to the generalised Elliott--Halberstam conjecture, labeled further as $GEH [\vartheta ]$ for some $0< \vartheta < 1$. This broad generalisation first appeared in \cite{GEH}. Its precise formulation can be found for example in \cite{Polymath8}. The best known result in this direction is currently proven by Motohashi \cite{Motohashi}.
\begin{thm}\label{GBV} $GEH[ \vartheta ]$ holds for every $\vartheta \in (0,1/2)$.
\end{thm}

In this work we actually need only one specific corollary of $GEH$, which can be perceived as an `Elliott--Halberstam conjecture for almost primes'. 

\begin{thm}\label{GEHwniosek} Assume $GEH[\vartheta]$. Let $r \geq 1$, $\epsilon > 0$, and $A\geq 1$ be fixed. Let 
\[ \Delta_{r,\epsilon} = \{ (t_1,\dots,t_r) \in [\epsilon,1]^r \colon ~ t_1 \leq \dots \leq t_r;~ t_1+\dots+t_r=1\}, \]
and let $F \colon \Delta_{r,\epsilon} \rightarrow {\bf R}$ be a fixed smooth function. Let $f \colon {\bf N} \rightarrow {\bf R}$ be the function defined by setting
\[\displaystyle  f (n) = F \left( \log_n p_1, \dots, \log_n p_r \right) \]
whenever $n=p_1 \dots p_k$ is the product of $r$ distinct primes $p_1 < \dots < p_r$ with $p_1 \geq x^\epsilon$ for some fixed $\epsilon>0$, and $f(n)=0$ otherwise. Then for every $Q \ll x^\vartheta$, we have
\[ \displaystyle  \sum_{q \leq Q} \max_{\substack{(a,q)=1}}
 \left| \Delta \left( \mathbf{1}_{[1,x] } f ; a \bmod q \right)   \right|
  \ll x \log^{-A} x.\]
\end{thm}

\section{Outline of the key ingredients}

Let us start from presenting a minor variation of \cite[Theorem 3.6]{Polymath8}. The only change we impose is replacing the linear forms of the shape $n+h_i$ by slightly more general $L_i(n)$. This, however, does not affect the proof in any way. 

\begin{prop}[Non-$\Omega$ sums]\label{Easy_Proposition}
Let $k \geq 1$ be fixed, let $\{ L_1, \dots , L_k \}$ be
a fixed admissible k-tuple, and let $b \bmod W$ be such that $(L_i (b) , W)=1$ for each
$i = 1, \dots , k$. For each fixed $1 \leq i \leq k$, let $F_i, \, G_i \colon [0,+\infty) \rightarrow \mathbf{R}$ be fixed smooth 
compactly supported functions. Assume one of the following hypotheses: \\
\begin{enumerate}
\item (Trivial case) One has  \[ \sum_{i=1}^k (S(F_i) + S(G_i)) < 1.\]  
\item (Generalized Elliott--Halberstam) There exist a fixed $0 < \vartheta < 1$ and $i_0 \in \{1, \dots , k\}$
such that $GEH[ \vartheta ]$ holds, and \[ \sum_{\substack{ 1 \leq i \leq k \\ i \not=i_0 }} (S(F_i) + S(G_i)) < \vartheta. \]
\end{enumerate}
Then, we have 
\[ \sum_{\substack{ n \sim x \\ n \equiv b \bmod W}}\prod_{i=1}^k \lambda_{F_i} (L_i (n)) \lambda_{G_i} (L_i (n) ) = (c+o(1)) B^{-k} \frac{x}{W}, \] 
where
\[ c :=\prod_{i=1}^k \left(  \int \limits_0^1 F_i'(t_i) \, G_i'(t_i) \, dt_i \right) . \]
\end{prop}

The next result is a crucial component of this work and is a novelty in the topic. Together with Proposition \ref{Easy_Proposition} it creates a way to transform $\mathcal{S}_0$ and $\mathcal{S}_\Omega$ into integrals, effectively converting the main task of finding almost primes into an optimization problem.

\begin{prop}[Sums containing $\Omega$ function]\label{Powerful_Proposition}
Let $k \geq 1$ and $i_0 \in \{ 1 , \dots , k\}$ be fixed, let $\{ L_1, \dots , L_k \}$ be
a fixed admissible k-tuple, and let $b \bmod W$ be such that $(L_i (b) , W)=1$ for each
$i = 1, \dots , k$. For each fixed $1 \leq i \leq k$, let $F_i, \, G_i,  \colon [0,+\infty) \rightarrow \mathbf{R}$ be fixed smooth 
compactly supported functions, and let  $\Upsilon \colon [0,+\infty) \rightarrow \mathbf{R}$ be a bounded Riemann integrable function continuous at $1$. Assume that there exist $\vartheta, \, \vartheta_0 \in (0,1)$ such that one of the following hypoteses holds: 
\begin{enumerate}
\item (Trivial case) One has  \[ \sum_{i=1}^k (S(F_i) + S(G_i)) <  1- \vartheta_0~~~~and~~~~S(\Upsilon) < \vartheta_0 . \]  
\item (Generalized Elliott--Halberstam) Assume that $GEH[ \vartheta ]$ holds, and
\[ \sum_{\substack{ 1 \leq i \leq k \\ i \not=i_0 }} (S(F_i) + S(G_i)) < \vartheta. \]
\end{enumerate}
Then, we have
\begin{equation}\label{upsilon}
  \sum_{ \substack{ n \sim x \\ n \equiv b \bmod W \\ \mathcal{P}(n)~ \textup{sq-free} }} \left(  \sum_{ p|L_{i_0} (n) }    \Upsilon (\log_x p) \right)
 \prod_{i=1}^k  \lambda_{F_{i}} ( L_i (n))  \lambda_{G_{i}} ( L_i (n)) =
  (c+o(1)) B^{-k} \frac{x}{W},
  \end{equation}
where
\[ 
\begin{split} c:=   \left(  \Upsilon (1) \, F_{i_0}(0) \, G_{i_0}(0)  ~+~  \int    \limits_0^1 \frac{\Upsilon ( y)}{y}  \int  \limits_0^{1-y}  \partial_{y} F'_{i_0} ( t_{i_0} )
\, \partial_{y} G'_{i_0}  ( t_{i_0} )  \, dt_{i_0} \, dy \right) 
  \prod_{\substack{ 1 \leq i \leq k \\ i \not=i_0 }} \left(  \int \limits_0^1 F_i'(t_i) \, G_i'(t_i) \, dt_i \right) .
  \end{split} \]
\end{prop}

 The first case of Proposition \ref{Powerful_Proposition} is strongly related to \cite[Proposition 5.1]{MaynardK} and \cite[Proposition 1.13]{Lewulis}. 
It is worth mentioning that the conditional results in the latter of these two cited papers relied only on $GEH[2/3]$. It was not possible to invoke the full power of $GEH$ by methods studied there due to certain technical obstacles. The second case of Proposition \ref{Powerful_Proposition} is strong enough to overcome them. It also paves a way to conveniently apply a device called an $\varepsilon$-trick in the unconditional setting.
 
 The role of the last proposition in this section is to deal with the contribution from $n$ such that $\mathcal{P} (n)$ is not square-free. 
 
 \begin{prop}[Sums with double prime factors]\label{double_prime_factors}
Let $k \geq 1$ be fixed, let $\{ L_1, \dots , L_k \}$ be
a fixed admissible k-tuple, and let $b \bmod W$ be such that $(L_i (b) , W)=1$ for each
$i = 1, \dots , k$. For each fixed $1 \leq i \leq k$, let $F_i, \, G_i \colon [0,+\infty) \rightarrow \mathbf{R}$ be fixed smooth 
compactly supported functions. Then, we have 
\[ \sum_{ \substack{ n \sim x \\ n \equiv b \bmod W \\ \mathcal{P}(n) \emph{ not sq-free} }} \tau ( \mathcal{P} (n))    
\left| \prod_{i=1}^k  \lambda_{F_{i}} ( L_i (n))  \lambda_{G_{i}} ( L_i (n)) \right| = o(1) \times B^{-k} \frac{x}{W}. \]
\end{prop}


Now, we combine Propositions \ref{Easy_Proposition}--\ref{double_prime_factors} to obtain Theorems \ref{st_simplex_sieving}, \ref{ext_simplex_sieving}, and \ref{eps_simplex_sieving} giving us criteria for the $DHL_\Omega$ problem. Theorem \ref{st_simplex_sieving} refers to sieving on standard simplex $\mathcal{R}_k$, which can be considered as a default range for the multidimensional Selberg sieve. The next one, Theorem \ref{ext_simplex_sieving}, deals with the extended simplex $\mathcal{R}_k'$, which was applied in \cite{Lewulis}, where $DHL_\Omega [5;14]$ was proven. We also prove Theorem \ref{eps_simplex_sieving} being the most general of these three. It describes sieving on the epsilon-enlarged simplex. In fact, Theorems \ref{st_simplex_sieving} and \ref{ext_simplex_sieving} are corollaries from Theorem \ref{eps_simplex_sieving}, as noted in Remark \ref{most_general}. 

\begin{thm}[Sieving on a standard simplex]\label{st_simplex_sieving}
 Suppose that there is an arbitrarily chosen fixed real parameter $\ell$ and a fixed $ \theta \in  (0,\frac{1}{2})$ such that $GEH[ 2 \theta ]$ holds.
Let $k \geq 2$ and $m \geq 1$ be fixed integers.
For any fixed compactly supported square-integrable function $F \colon [0,+\infty )^k \rightarrow \mathbf{R}$, define the functionals
\begin{equation}
\begin{split}
I (F) : =& \int_{[0,+\infty )^k}  F (t_1, \dots , t_k)^2   \, dt_1 \dots dt_k, \\
Q_i (F) :=&   \int_0^{\frac{1}{\theta }}  \frac{1- \ell \theta y}{y} \int_{[0,+\infty)^{k-1}}   \left(  
  \int_0^{\frac{1}{\theta } - y}\left(  \partial_y^{(i)} F (t_1, \dots , t_k) \right)^2   dt_i  \right)  \, dt_1 \dots dt_{i-1} \, dt_{i+1} \dots dt_k \, dy,  \\ 
J_i (F) :=&   \int_{[0,+\infty)^{k-1} } \left( \int_0^\infty   F(t_1, \dots , t_k) \, dt_i \right)^2 dt_1 \dots dt_{i-1} \, dt_{i+1} \dots dt_k,
\end{split}
\end{equation}
and let $\Omega_k$ be the infimum
\begin{equation}\label{Omega_k}
\Omega_k := \inf_F   \left( \frac{ \sum_{i=1}^k \left( Q_i (F) + \theta(1-\ell ) J_i (F) \right) }{I(F)} +\ell k \right) , 
\end{equation}
over all square integrable functions $F$ that are supported on the simplex
\[ \mathcal{R}_k := \{ (t_1, \dots , t_k) \in [0,+\infty )^k \colon t_1 + \dots + t_k \leq 1  \} , \] 
  and are not identically zero up to almost everywhere equivalence. If
\[ m > \Omega_k, \]
then $DHL_\Omega [k; m-1]$ holds.
\end{thm}

\begin{remark} Due to the continuity of $\Omega_k$ we can replace the condition that $GEH [2\theta ]$ holds by a weaker one that $GEH[2 \theta' ]$ holds for all $\theta' < \theta$. Therefore, we are also permitted to take $\theta=1/4$ unconditionally and $\theta = 1/2$ assuming $GEH$. The same remark also applies to Theorems \ref{1dim_sieving}, \ref{ext_simplex_sieving}, and \ref{eps_simplex_sieving}.
\end{remark}

The choice of parameter $\ell$ does not affect the value of $\Omega_k$.  Substituting
\[ F(t_1, \dots , t_k) = f(t_1 + \dots + t_k) \]
for some $f \colon [0, +\infty ) \rightarrow \mathbf{R}$ and fixing $\ell=1$ we get the following result.

\begin{thm}[One-dimensional sieving]\label{1dim_sieving}
 Suppose that there is a fixed $ \theta \in  (0,\frac{1}{2})$ such that $GEH[ 2 \theta ]$ holds.
Let $k \geq 2$ and $m \geq 1$ be fixed integers. For any fixed and locally square-integrable function $f \colon [0,+\infty ) \rightarrow \mathbf{R}$, define
the functionals
\begin{equation}\label{functionals_1dim}
\begin{split}
\bar{I} (f) : =& \int \limits_0^1 f (t)^2 \, t^{k-1} \, dt, \\
\bar{Q}^{(1)} (f) :=& \int \limits_0^1  \frac{1- \theta y}{y} \int \limits_0^{1-y} \left( f(t) - f(t+y) \right)^2 t^{k-1} \, dt \, dy,  \\
\bar{Q}^{(2)} (f) :=& \left( \int \limits_0^1   \int \limits_{1-y}^1 \, + \,  \int \limits_1^{\frac{1}{\theta} -1}  \int \limits_{0}^1   
\, + \, \int \limits_{\frac{1}{\theta} -1}^{\frac{1}{\theta}}  \int \limits_{0}^{\frac{1}{\theta} - y} \, \right) \frac{1- \theta y}{y} \, f(t)^2 \, t^{k-1} \, dt \, dy,  \\
\bar{Q}^{(3)} (f) :=& \int \limits_{\frac{1}{\theta} -1}^{\frac{1}{\theta}}  \frac{1- \theta y}{y} \int \limits_{\frac{1}{\theta} - y}^1 f(t)^2 \, \left( t^{k-1} - \left(t+y - \frac{1}{\theta} \right)^{k-1} \right)  dt \, dy,  
\end{split}
\end{equation}
and let $\bar\Omega_k$ be the infimum
\[ \bar\Omega_k := \inf_f   \left( \frac{ \sum_{i=1}^3 \bar{Q}^{(i)} (f)}{\bar{I}(f)} +1 \right) \cdot k, \]
over all square integrable functions $f$ that are not identically zero up to almost everywhere equivalence. If
\[ m > \bar\Omega_k, \]
then $DHL_\Omega [k; m-1]$ holds.
\end{thm}

We obviously have $\bar{\Omega}_k \geq \Omega_k$ for every possible choice of $k$. We may apply Theorem \ref{1dim_sieving} to get some  non-trivial improvements over the current state of the art in the $GEH$ case. We perform optimization over polynomials of the form $f(x)=a+b(1-x)+c(1-x)^2+d(1-x)^3$ for $-1 < a,b,c,d < 1$. This choice transforms the functionals (\ref{functionals_1dim})  into quadratic forms depending on the parameters $a,b,c,d$. Details including close to optimal polynomials (up to a constant factor) for each $k$ are covered in the table below.

\newpage

\begin{center} 
\centering
\text{Table C. Upper bounds for $\Omega_k$.}
\vspace{1mm}
\\
\renewcommand{\arraystretch}{1}
  \begin{tabular}{ | D || E |  E  | G |@{}m{0cm}@{}}\hline
    $k$ &  $\theta = 1/4 $ & $\theta = 1/2 $  &  $f(1-x)$ &  \rule{0pt}{3ex}  \\ \hline 
    $2$ & 5.03947 &  3.84763 &  $3 + 25x - x^2 + x^3$     &  \\
    $3$ & 8.15176  &  6.31954  & $1 + 12x - 2x^2 + 9x^3$    &  \\ 
    $4$ & 11.49211  & 9.00542 &  $1 + 15x - x^2 + 19x^3 $   &  \\ 
    $5$ & 15.01292  & 11.86400 &  $1 + 16x + 5x^2 + 32x^3 $   &  \\ 
    $6$ & 18.68514  & 14.86781 &  $1 + 26x - 8x^2 + 86x^3$   &  \\ 
    $7$ & 22.48318  & 17.99402 &  $1 + 24x + 6x^2 + 110x^3$   &  \\
    $8$ & 26.39648  & 21.23219 &  $1 + 30x + x^2 + 200x^3$   &  \\ 
    $9$ & 30.40952  & 24.56817 &  $1 + 30x + 3x^2 + 260x^3$   &   \\ 
    $10$ & 34.51469  & 27.99372  & $1 + 36x - x^2 + 400x^3$  &   \\
      \hline
  \end{tabular}
\end{center}
It turns out that close to optimal choices in the unconditional setting are also close to optimal under $GEH$. These results are sufficient to prove the conditional part of Theorem \ref{MAIN} in every case except for $k=4$. Unfortunately, by this method we cannot provide any unconditional improvement over what is already obtained in  \cite{MaynardK}, as presented in Table C. Therefore, let us try to expand the sieve support a bit. 

\begin{thm}[Sieving on an extended simplex]\label{ext_simplex_sieving}
 Suppose that there is a fixed $ \theta \in  (0,\frac{1}{2})$ such that $GEH[ 2 \theta ]$ holds and an arbitrarily chosen fixed real parameter $\ell$.
Let $k \geq 2$ and $m \geq 1$ be fixed integers. Let $\Omega_k^{\emph{ext}}$ be defined as in (\ref{Omega_k}), but where the supremum now ranges over all square-integrable and non-zero up to almost everywhere equivalence $F$ supported on the extended simplex
\[  \mathcal{R}'_k := \{ (t_1, \dots , t_k) \in [0,+\infty )^k \colon  \forall_{i \in \{ 1, \dots , k \} }~ t_1 + \dots + t_{i-1} + t_{i+1} + \dots + t_k \leq 1   \}.\]  
 If
\[ m > \Omega^\emph{ext}_k, \]
then $DHL_\Omega [k; m-1]$ holds.
\end{thm}

It is difficult to propose a one-dimensional variation of Theorem \ref{ext_simplex_sieving} in a compact form, because the precise shape of functionals analogous to (\ref{functionals_1dim}) varies depending on $k$. We deal with this problem in Subsection \ref{bounds_omegaextk}. Given that, we apply Theorem \ref{ext_simplex_sieving} directly and perform optimization over polynomials of the form $F(t_1,\dots,t_k)=a+b(1-P_1)+c(1-P_1)^2 =: f(P_1)$ for $-1 < a,b,c <1$. Our choice is motivated by the fact that the values of symmetric polynomials generated only by $P_1$ depend only on the sum $t_1+\dots+t_k$, so they behave 'one-dimensionally', which makes all necessary calculations much easier. Moreover, our numerical experiments suggest that including $P_2$ does not provide much extra contribution. Some good choices of polynomials (again, up to a constant factor) and the bounds they produce are listed below.

\newpage

\begin{center} 
\centering
\text{Table D. Upper bounds for $\Omega^{\text{ext}}_k$.}
\vspace{1mm}
\\
\renewcommand{\arraystretch}{1}
  \begin{tabular}{ | D || E |  E  | J |@{}m{0cm}@{}}\hline
    $k$ &  $\theta = 1/4 $ & $\theta = 1/2 $  &  $f(1-x)$ &  \rule{0pt}{3ex}  \\ \hline 
    $2$ & 4.49560 &  3.35492 &  $6 + 8x + 3x^2$     &  \\
    $3$ & 7.84666  &  6.03889  & $2 + 7x + 7x^2$    &  \\ 
    $4$ & 11.27711  & 8.80441 &  $1 + 6x + 9x^2$   &  \\ 
    $5$ & 14.84534  & 11.70582 &  $1 + 7x + 15x^2$   &  \\ 
    $6$ & 18.55409  & 14.74036 &  $1 + 9x + 32x^2$   &  \\ 
    $7$ & 22.38208  & 17.89601 &  $1 + 10x + 46x^2$   &  \\
    $8$ & 26.32546  & 21.16260 &  $1 + 10x + 65x^2$   &  \\ 
    $9$ & 30.37012  & 24.52806 &  $1 + 10x + 90x^2$   &   \\ 
    $10$ & 34.50669  & 27.98326  & $1 + 11x + 121x^2$  &   \\
      \hline
  \end{tabular}
\end{center}

The results from the $\theta = 1/4$ column in Table D predict the limitations of methods developed in \cite{Lewulis}. In the conditional case we also get a strong enhancement over what is achievable by sieving on the standard simplex in the $k=4$ case. In the $k=2$ case we observe a standard phenomenon that passing through the constant $3$ seems impossible, most probably because of the parity obstruction as mentioned in \cite{Polymath8}. In this work we do not 
make any attempt to break this notorious barrier, so we do not expect to outdo the result of Chen -- even assuming very strong distributional claims like $GEH$.

In order to push our results even more, we would like to apply a device called an $\varepsilon$-trick, which made its debut in \cite{Polymath8}. The idea is to expand the sieve support even further than before, but at a cost of turning certain asymptotics into lower bounds. This is also the place where the $\ell$ parameter starts to behave non-trivially.

\begin{thm}[Sieving on an epsilon-enlarged simplex]\label{eps_simplex_sieving}
 Suppose that there is a fixed $ \theta \in  (0,\frac{1}{2})$ such that $GEH[ 2 \theta ]$ holds, and arbitrarily chosen fixed real parameters $\ell > 1 $, $\varepsilon \in [0,1)$, and $\eta \geq 1+\varepsilon$ subject to the constraint
\begin{equation}\label{wiezy}
 2 \theta \eta  + \frac{1}{\ell} \leq 1.
  \end{equation}
Let $k \geq 2$ and $m \geq 1$ be fixed integers.
For any fixed compactly supported square-integrable function $F \colon [0,+\infty )^k \rightarrow \mathbf{R}$, define the functionals
\begin{equation}
\begin{split}
J_{i,\varepsilon} (F) :=&   \int_{(1-\varepsilon)\cdot \mathcal{R}_{k-1} } \left( \int_0^\infty   F(t_1, \dots , t_k) \, dt_i \right)^2 dt_1 \dots dt_{i-1} \, dt_{i+1} \dots dt_k, \\
 Q_{i,\varepsilon}(F):=&  \int_0^{\frac{1}{ \theta}}   \frac{1- \ell \theta y}{y}   \int_{ \Phi (y) \cdot \mathcal{R}_{k-1}  } \left(  \, \int_0^{\frac{1}{\theta } - y} \left(  \partial_y^{(i)} F (t_1, \dots , t_k) \right)^2  dt_i \right)  dt_1 \dots dt_{i-1} \, dt_{i+1} \dots dt_k   \, dy, 
\end{split}
\end{equation}
where $\Phi \colon [0,+\infty ) \rightarrow \mathbf{R}$ is a function given by the formula
\begin{equation}
\Phi (y) := 
\begin{cases}
1+ \varepsilon, & \emph{for } y \in \left[ 0 ,  \frac{1}{\ell \theta} \right) , \\
1 - \varepsilon, & \emph{for } y \in \left[ \frac{1}{\ell \theta}  , \frac{1}{\theta } \right] , \\
0, & \emph{otherwise.}
\end{cases}
\end{equation}
Let $\Omega_{k,\varepsilon}$ be the infimum
\begin{equation}\label{funkcjonal_eps}
 \Omega_{k,\varepsilon} := \inf_{\eta, F}   \left( \frac{ \sum_{i=1}^k \left( Q_{i,\varepsilon}  (F)   - \theta(\ell -1 ) J_{i,\varepsilon} (F) \right) }{I(F)} +\ell k \right) , 
\end{equation} 
over all square integrable functions $F$ that are supported on the region
\[ (1 + \varepsilon ) \cdot \mathcal{R}_k' \, \cap \, \eta \cdot \mathcal{R}_k , \] 
  and are not identically zero up to almost everywhere equivalence. If
\[ m > \Omega_{k,\varepsilon}, \]
then $DHL_\Omega [k; m-1]$ holds. Moreover, if $\varepsilon =0$, then constraint (\ref{wiezy}) can be discarded and the functional inside the parentheses in (\ref{funkcjonal_eps}) is constant with respect to the $\ell$ variable. 
\end{thm}

\begin{remark}\label{most_general} Observe that Theorems \ref{st_simplex_sieving} and \ref{ext_simplex_sieving} follow easily from Theorem \ref{eps_simplex_sieving}. In the first case we just consider $\varepsilon=0$ and $\eta=1$. To prove the latter, we take the same $\varepsilon$ and any $\eta \geq k/(k-1)$.
\end{remark}

Constraint (\ref{wiezy}) refers to the hypotheses mentioned in the `trivial case' from Proposition \ref{Powerful_Proposition}. Notice that we do not have to restrict the support of the $Q_{i,\varepsilon} $ integrals for $ y \in \left[ 0 ,  \frac{1}{\ell \theta} \right)$, because we do not apply any $EH$-like theorem/conjecture in this interval. Below we present some upper bounds for $\Omega_{k,\varepsilon}$ obtained via considering $\eta=1+\varepsilon$ and optimizing over polynomials of the form $a+b  (1- P_1) + c  (1-P_1)^2$ for $-1<a,b,c<1$ supported on the simplex $(1+\varepsilon) \cdot \mathcal{R}_k$:
\begin{center} 
\centering
\text{Table E. Upper bounds for $\Omega_{k, \varepsilon}$.}
\vspace{1mm}
\\
\renewcommand{\arraystretch}{1}
  \begin{tabular}{ | D || D |  E  |@{}m{0cm}@{}}\hline
    $k$ &  $ \varepsilon $ & $\theta = 1/4 $  & \rule{0pt}{3ex}  \\ \hline 
    $2$ &  1/3   & 4.69949 & \\
    $3$ &  1/4   & 7.75780   &  \\ 
    $4$ &  1/5  &  11.05320 & \\ 
    $5$ &  1/6  &  14.54134  & \\ 
    $6$ &  1/7  &  18.19060 & \\ 
    $7$ &  1/9 &   21.99368 & \\ 
    $8$ &  1/10 &  25.90287 & \\ 
    $9$ &  1/10  &  29.90565 & \\ 
    $10$ &  2/21 & 34.01755  & \\ 
      \hline
  \end{tabular}
\end{center}

We are also able to obtain the bound $33.93473$ for $k=10$ and the same $\varepsilon$, if one optimizes over polynomials of the form $a+b  (1-P_1) + c (1-P_1)^2 + d (1-P_1)^3$ for $-1<a,b,c,d<1$. We observe that results provided by the $\varepsilon$-trick are considerably stronger than those listed in Table D for every $k\geq 3$. They surpass the currently known value of $\varrho_k$ in (\ref{twierdzenie_zasadnicze}) for $7 \leq k \leq 10$. Let us also notice that the bigger $k$ we take, the better improvement over Tables C and D we obtain. The reason for this is that the region $\mathcal{R}'_k$ is much larger than simplex $\mathcal{R}_k$ for small $k$, but the difference in size is far less spectacular for bigger values of $k$. In the same time, the epsilon-enlarged simplex $(1+\varepsilon)\cdot \mathcal{R}_k$ does not share this weakness.
 
 \begin{remark} It is possible to consider other choices of $\eta$ than $1+\varepsilon$. One of them is $(1+\varepsilon)k/(k-1)$, which gives an access to a larger domain $(1+\varepsilon) \cdot \mathcal{R}_k'$. However, expanding the sieve support so far makes the constaint (\ref{wiezy}) more restrictive. As for this moment, numerical experiments suggest that one loses more than wins by implementing such a manouver. The author also tried excluding the fragment
 \[   \{ (t_1, \dots , t_k) \in [0,+\infty )^k \colon \forall_{ i \in \{1, \dots , k \} } ~
 t_1 + \dots +t_{i-1} + t_{i+1} + \dots +  t_k > 1 - \varepsilon   \}, \] 
motivated by the fact, that it contibutes neither to $J_{i,\varepsilon} (F)$, nor the negative part of $Q_{i,\varepsilon} (F)$, and in the same time it contributes to $I (F)$. Unfortunately, this technique did not generate any substancial advantage.
\end{remark}


\section*{Lemmata}

We have the following lemma enabling us to convert certain sums into integrals. 

\begin{lemma}\label{sumynacalki}
Let $m \geq 1$ be a fixed integer and let $f \colon (0,+\infty)^m \rightarrow \bf{C}$ be a fixed compactly supported, Riemann integrable function. Then for $x>1$ we have
\[ \sum_{\substack{ p_1, \dots , p_m \\ p_1 \cdots p_m \sim x}} \, f \left( \log_x p_1 , \dots , \log_x p_m \right)  =
 \left( c_f  + o (1) \right) 
 \frac{x}{\log x},\]
where
\[c_f :=  \int_{\substack{  t_1+\dots+t_m=1}} f(t_1, \dots , t_m) \frac{dt_1 \dots dt_{m-1}}{t_1\cdots t_m}, \]
 where we lift Lebesgue measure $dt_1 \dots dt_{m-1}$ up to the hyperplane $t_1 + \cdots + t_m = 1$.
 \end{lemma}
\begin{proof} Follows from prime number theorem combined with elementary properties of the Riemann integral.
\end{proof}

 We introduce an another useful lemma which helps us with discarding those $n \sim x$ having low prime factors. 

\begin{lemma}[Almost primality]\label{Almost primality}
 Let $k \geq 1$ be fixed, let $(L_1, \dots  , L_k)$ be a fixed admissible
$k$--tuple, and let $b \bmod W$ be such that $(L_i(b), W)=1$ for each $i = 1, \dots , k$.
Let further $F_1, \dots , F_k \colon [0,+\infty ) \rightarrow \mathbf{R}$ be fixed smooth compactly supported functions, and let
$m_1, \dots ,m_k \geq 0$ and $a_1, \dots , a_k \geq 1$ be fixed natural numbers. Then,
\[ \sum_{\substack{ n \sim x \\ n \equiv b \bmod W }} \prod_{j=1}^k
\left( \left| \lambda_{F_j} (L_j(n)) \right|^{a_j} \tau ( L_j (n) )^{m_j} \right) \ll B^{-k} \frac{x}{W}.\]
Furthermore, if $1 \leq j_0 \leq k$ is fixed and $p_0$ is a prime with $p_0 \leq x^{1/10k}$, then we have the
variant
\[ \sum_{\substack { n \sim x \\ n \equiv b \bmod W }} \prod_{j=1}^k 
\left( \left| \lambda_{F_j} (L_j (n)) \right|^{a_j} \tau (L_j (n) )^{m_j}  \right) \mathbf{1}_{p_0| L_{j_0} (n)} \ll
\frac{ \log_x p_0 }{p_0} B^{-k} \frac{x}{W}. \]
As a consequence, we have
\[ \sum_{\substack{ n \sim x \\ n \equiv b \bmod W }} \prod_{j=1}^k 
\left( \left| \lambda_{F_j} (L_j (n))\right|^{a_j} \tau (L_j (n) )^{m_j}  \right) \mathbf{1}_{ \textup{lpf}(L_{j_0} (n) ) \leq x^{\epsilon} } \ll
\epsilon B^{-k} \frac{x}{W}, \]
for any $\epsilon > 0$. 
\end{lemma}
\begin{proof}
This is a trivial modification of \cite[Proposition 4.2]{Polymath8}.
\end{proof}

\section{Proof of Propositions \ref{Powerful_Proposition} and \ref{double_prime_factors} }

Contraty to the numerical ordering, we tackle Proposition \ref{double_prime_factors} first, because it is going to be needed throughout the rest of this section.

 \subsection*{Propositon \ref{double_prime_factors}}
 
\begin{proof}
It suffices to show that 
 \begin{equation}\label{rozbicie_pierwsze}
  \sum_{p}  \, \sum_{ \substack{ n \sim x \\ n \equiv b \bmod W \\ p^2 | \mathcal{P}(n)  }} \tau ( \mathcal{P} (n))    
 \left| \prod_{i=1}^k  \lambda_{F_{i}} ( L_i (n))  \lambda_{G_{i}} ( L_i (n)) \, \right| \, = \, o(1) \times  B^{-k} \frac{x}{ W }. 
 \end{equation}
 Choose an $\epsilon >0$. We decompose the outer sum in (\ref{rozbicie_pierwsze}) as follows:
 \begin{equation}\label{decompos}
   \sum_{p}   ~ = ~
      \sum_{p \leq x^\epsilon} ~+~    \sum_{p > x^\epsilon }.
 \end{equation}
We apply the divisor bound $\tau (n) \ll n^{o(1)}$, valid for all $n \in \mathbf{N}$, to conclude that the second sum from the right-hand side of (\ref{decompos}) is
\begin{equation}
\ll x^{o(1)} \sum_{p > x^\epsilon} \sum_{\substack{ n \sim x \\ p^2 | \mathcal{P}(n) }} 1 \ll x^{1 - \epsilon + o(1)}.
\end{equation}
The first sum, by the third part of Lemma \ref{Almost primality}, can be easily estimated as being 
\[ \ll  ~ \epsilon B^{-k} \frac{x}{W} . \]
To this end, we only have to send $\epsilon \rightarrow 0$ sufficently slowly. 
\end{proof}

\subsection*{The trivial case of Proposition \ref{Powerful_Proposition}}

\begin{proof}
We shall take $i_0=k$, as the other cases can be proven exactly the same way. Proposition \ref{double_prime_factors} implies that our task is equivalent to showing that 
\begin{equation}\label{teza_triv_case}
  \sum_{\substack{ n \sim x \\ n \equiv b \bmod W   }}
   \sum_{p|L_k (n)} \Upsilon (\log_x p ) 
 \prod_{i=1}^k  \lambda_{F_{i}} ( L_i (n))  \lambda_{G_{i}} ( L_i (n)) =
  (c +o(1)) B^{-k} \frac{x}{W}.
  \end{equation}
  Interchanging the order of summation, we get that the left-hand side of (\ref{teza_triv_case}) equals
  
    \begin{equation}\label{<U_2}
   \sum_{p }  \Upsilon (\log_x p ) 
 \sum_{\substack{ d_1, \dots , d_k \\ e_1, \dots , e_k }} \left( \prod_{i=1}^k \mu (d_i) \mu (e_i) F_i ( \log_x d_i) G_i (\log_x e_i)  \right) 
 S_p (d_1, \dots , d_k, e_1, \dots , e_k) ,
  \end{equation}
 where 
 \begin{equation}\label{S_eps}
 S_p (d_1, \dots , d_k, e_1, \dots , e_k)  :=   \sum_{\substack{ n \sim x \\ n \equiv b \bmod W \\    \forall_i \, [d_i, e_i] | L_i (n) \\ p|L_k (n) }} 1.
 \end{equation}
 By hypotheses, all the $L_i (n)$ are coprime to $W$. We also assumed that for all distinct $i, \, j$ we have $| A_i B_j - A_j B_i | < D_0$. On the other hand, if there exists a prime $p_0$ dividing both $[d_i, e_i]$ and $[d_j, e_j]$, then $ A_i B_j - A_j B_i \equiv 0 \bmod p_0$, which forces $p_0 \leq D_0$. By this contradiction, we may further assume in this subsection that $W, \, [d_1, e_1],  \dots ,  [d_k,e_k]$ are pairwise coprime, because otherwise $S_p$ vanishes. We mark this extra constraint by the  $'$ sign next to the sum (see (\ref{prim_sum}) for an example). Under these assumptions, we can can merge the congruences appearing under the sum in (\ref{S_eps}) into one:
\begin{equation}
n \equiv a \bmod q ,
 \end{equation} 
 where 
 \begin{equation}
q  := W\,  [d_k,e_k,p] \prod_{i=1}^{k-1} [d_i, e_i]
 \end{equation}
and $(a,q)=1$. This gives
\begin{equation}\label{zamiana_modulow}
 S_p (d_1, \dots , d_k, e_1, \dots , e_k)  =   \sum_{\substack{ n \sim x  \\  n \equiv a \bmod q   }} 1 \, = \,
 \frac{x}{q} + O(1) .
\end{equation}
The net contribution of the $O(1)$ error term to (\ref{<U_2}) is at most 
\begin{equation}
\ll  \, \left( \sum_{d,e \leq x} \frac{1}{[d,e]} \right)^{k-1}  \sum_{\substack{ d,e,p \leq x  }} \frac{1}{[d,e,p]} \ll \left( \sum_{r \leq x} \frac{ \tau (r)^{O(1)} }{r} \right)^k \leq x^{o(1)}.
\end{equation}
Therefore, it suffices to show that 
\begin{equation}\label{prim_sum}
 \sum_{p } \frac{ \Upsilon (\log_x p ) }{p}
 \left(  \prod_{i=1}^{k}  \sideset{}{'}\sum_{d_i, e_i } \frac{ \mu (d_i) \mu (e_i) F_i ( \log_x d_i) G_i (\log_x e_i) }{ \psi_i( [d_i, e_i]) } \right)    = (c +o(1)) B^{-k} ,
\end{equation}
where
\begin{equation} \psi_i (n) := 
\begin{cases}
n, & \text{for } i \in \{1, \dots , k-1 \}  , \\
[n,p]/p , & \text{for } i=k.
\end{cases}
\end{equation}
By \cite[Lemma 2.2 and Lemma 2.6]{Lewulis} and the polarization argument we get 
\begin{equation}\label{lemma_40}
\prod_{i=1}^{k} \sideset{}{'}\sum_{d_i, e_i } \frac{ \mu (d_i) \mu (e_i) F_i ( \log_x d_i) G_i (\log_x e_i) }{ \psi_i( [d_i, e_i]) } =
     (c'c'' + o(1)) B^{-k} , 
\end{equation}
with
\begin{align}
c' &:=  \prod_{i=1}^{k-1} \int_0^1 F'_i (t) G'_i (t) \, dt, \\
c'' &:=  \int_0^{1- \log_x p }  \partial_{y} F'_k ( t)
\, \partial_{y} G'_k  ( t  )  \, dt \, dy. 
\end{align}
\begin{remark}
To justify this application, we need to consider (under the notation used within the cited work) 
\[ \lambda_{d_1 , \dots , d_k} := \prod_{i=1}^k \mu (d_i) \widetilde{F}_i (\log_x d_i) \]
in one case and 
\[ \lambda_{d_1 , \dots , d_k} := \prod_{i=1}^k \mu (d_i) \widetilde{G}_i (\log_x d_i) \]
in the other -- we are permitted to choose these weights arbitrarily due to \cite[Lemma 1.12]{Lewulis}. The key relationship in that paper between
 $\lambda_{d_1 , \dots , d_k} $ and $y_{r_1, \dots  r_k}$ may be established via \cite[(1.20) and Lemma 2.6]{Lewulis}. Then, from a simple formula 
\[ \widetilde{F}^2 - \widetilde{G}^2 = (  \widetilde{F} - \widetilde{G} ) ( \widetilde{F} + \widetilde{G} )  \] 
we deduce that after defining $\widetilde{F}$, $\widetilde{G}$ in such a way that  $F=\widetilde{F} - \widetilde{G}$ and $G=\widetilde{F} + \widetilde{G}$, and comparing the two mentioned choices of $\lambda_{d_1, \dots, d_k}$, our argument is completed.
\end{remark}
 The expression $1- \log_x p$ in the upper limit of the integral may seem a bit artificial. Its role is to unify this part of Proposition \ref{Powerful_Proposition} with the second one. 
Now, it suffices to show that 
\begin{equation}
 \sum_{p } \frac{ \Upsilon (\log_x p ) }{p}  \int_0^{1- \log_x p }  \partial_{y} F'_k ( t)
\, \partial_{y} G'_k  ( t  )  \, dt  = 
 \int_0^1 \frac{\Upsilon ( y)}{y}  \int_0^{1-y}  \partial_{y} F'_k ( t_k )
\, \partial_{y} G'_k  ( t_k )  \, dt_k \, dy .
\end{equation}
This is a direct application of Lemma \ref{sumynacalki}.
\end{proof}

\subsection*{The Elliott--Halberstam case of Proposition \ref{Powerful_Proposition}}
\begin{proof}

As in the previous subsection, we can take $i_0=k$ without loss of generality. Again, by Proposition \ref{double_prime_factors} we have to prove that 
\begin{equation}\label{zamiana_omeg_J0}
  \sum_{\substack{ n \sim x \\ n \equiv b \bmod W   }}
   \sum_{p|L_k (n)} \Upsilon (\log_x p ) 
 \prod_{i=1}^k  \lambda_{F_{i}} ( L_i (n))  \lambda_{G_{i}} ( L_i (n)) =
  (c +o(1)) B^{-k} \frac{x}{W}.
  \end{equation}
Take some $\epsilon > 0$. We decompose the studied sum as follows: 
\begin{equation}\label{epsilon_decomposition}
\sum_{ \substack{ n \sim x \\ n \equiv b \bmod W  }}  = 
 \sum_{\substack{ n \sim x \\ n \equiv b \bmod W  \\ \text{lpf} (L_k (n)) \leq x^{\epsilon}   }} + 
  \sum_{\substack{ n \sim x \\ n \equiv b \bmod W  \\ \text{lpf} (L_k (n)) > x^{\epsilon}  }} .
\end{equation}
We show that the contribution of the first sum from the right-hand side of (\ref{epsilon_decomposition}) is $\ll \epsilon B^{-k} x W^{-1}$.  To do so we bound
\begin{equation}\label{CS_Omega<U}
 \lambda_{F_{i}} ( L_i (n))   \lambda_{G_{i}} ( L_i (n)) \leq  \frac{1}{2} \left(  \lambda_{F_i} ( L_i (n))^2 + \lambda_{G_{i}} ( L_i (n))^2 \right)
\end{equation}
for each $i=1, \dots , k$. We also recall the trivial inequality
\begin{equation}\label{TrivOmega<U}
 \sum_{p|L_k (n)} \Upsilon (\log_x p )  \ll \tau (L_k (n)). 
\end{equation}
By (\ref{CS_Omega<U}) and (\ref{TrivOmega<U}) we can present the first sum from the right-hand side of (\ref{epsilon_decomposition}) as a linear combination of sums that can be threated straightforwardly by Lemma \ref{Almost primality}.

Let us define a function
\[ \Omega^\flat (n) :=  \sum_{\substack{p|n \\ p>x^\epsilon }}  \Upsilon (\log_x p )   \]
Now, it sufficies to show that for any $\epsilon >0 $ we have
\begin{equation}\label{<U_1p}
  \sum_{\substack{ n \sim x \\ n \equiv b \bmod W \\   \text{lpf} (L_k (n)) > x^{\epsilon}  }}
   \Omega^\flat ( L_k (n) )    
 \prod_{i=1}^k  \lambda_{F_{i}} ( L_i (n))  \lambda_{G_{i}} ( L_i (n)) =
  (c_\epsilon +o(1)) B^{-k} \frac{x}{W},
  \end{equation}
  where $c_\epsilon \rightarrow c$ when $\epsilon \rightarrow 0$. After expanding the $\lambda_{F_i}, \, \lambda_{G_i}$ we conclude that the left-hand side of (\ref{<U_1p}) equals
  \begin{equation}
 \sum_{\substack{ d_1, \dots , d_{k-1} \\ e_1, \dots , e_{k-1} }} \left( \prod_{i=1}^{k-1} \mu (d_i) \mu (e_i) F_i ( \log_x d_i) G_i (\log_x e_i)  \right) 
 S_{\epsilon} (d_1, \dots , d_{k-1}, e_1, \dots , e_{k-1}) ,
  \end{equation}
 where 
 \begin{equation}\label{S_eps2}
S_{\epsilon } (d_1, \dots , d_{k-1}, e_1, \dots , e_{k-1})  :=   \sum_{\substack{ n \sim x \\ n \equiv b \bmod W \\   \text{lpf} (L_k (n)) > x^{\epsilon} \\ \forall_{i\not= k} \, [d_i, e_i] | L_i (n)  }} \Omega^\flat ( L_k (n) )  \, \lambda_{F_{k}} ( L_k (n))  \, \lambda_{G_{k}}  ( L_k (n)) .
 \end{equation}
 Notice that $n \equiv b \bmod W$  implies that all of the $L_i (n)$ are coprime to $W$. We also assumed that for all distinct $i, \, j$ we have $| A_i B_j - A_j B_i | < D_0$, so if there exists a prime $p_0$ dividing both $[d_i, e_i]$ and $[d_j, e_j]$, then 
$ A_i B_j - A_j B_i \equiv 0 \bmod p_0,$
  which forces $p_0 \leq D_0$. That is a contradiction. Therefore, we may further assume in this subsection that $W, \, [d_1, e_1],  \dots ,  [d_k,e_k]$ are pairwise coprime and that $\text{lpf} \, ([d_k, e_k]) > x^\epsilon$, because otherwise $S_{\epsilon}$ vanishes. Under these assumptions we can merge all the congruences under the sum (\ref{S_eps2}) into two: 
\begin{equation}
n \equiv a \bmod q \, , ~~~~L_k(n) \equiv 0 \bmod [d_k, e_k, p], 
 \end{equation} 
 where we redefine $q$ and $a$ as
 \begin{equation}
q := W \prod_{i=1}^{k-1} [d_i, e_i],
 \end{equation}
and $a$ being some residue class coprime to its modulus such that $(L_i( a ), W)=1$ for each possible choice of index $i$. This gives
\begin{equation}\label{zamiana_modulow2}
S_{\epsilon} (d_1, \dots , d_{k-1}, e_1, \dots , e_{k-1})  =   \sum_{\substack{ n \sim x \\   \text{lpf} (L_k (n)) > x^{\epsilon} \\n \equiv a \bmod q   }}  \Omega^\flat ( L_k (n) )  \, \lambda_{F_{k}} ( L_k (n))  \, \lambda_{G_{k}}  ( L_k (n)) .
\end{equation}
We would like to perform a substitution $m:=L_k (n)$ in the sum from (\ref{zamiana_modulow2}), so we have to transform the congruence $n \equiv a \bmod q$ appropriately. In order to do so, we split it into two: $n \equiv a \bmod [A_k,q]/A_k$ and $n \equiv a \bmod \mbox{rad} \,A_k$, where $\mbox{rad} \,A_k$ denotes the square-free part of $A_k$. The former congruence is simply equivalent to $m \equiv L_k (a) \bmod [A_k,q]/A_k$. The latter is equivalent to $m \equiv L_k (a) \bmod A_k \, \text{rad} \,A_k$ and it also implies $m \equiv B_k \bmod A_k$, which has to be satisfied by our substitution. Note that
 \begin{equation}
 ( L_k (a) , [A_k,q]/A_k)=( L_k (a) , A_k \, \mbox{rad} A_k)=1,
\end{equation}
so we can combine the two considered congruences into one $m \equiv a' \bmod [A_k, q] \, \mbox{rad} A_k$. Hence,
\begin{equation}
 S_{\epsilon} (d_1, \dots , d_{k-1}, e_1, \dots , e_{k-1})   =  \sum_{\substack{ A_k x + B_k < m \leq 2A_k x  + B_k \\   \text{lpf} (m) > x^{\epsilon} \\m \equiv a' \bmod q'  }}  \Omega^\flat ( m )  \, \lambda_{F_{k}} ( m ) \,  \lambda_{G_{k}}  ( m ) ,
\end{equation}
where $q':=  [A_k, q] \, \mbox{rad} \, A_k = q A_k $ and $a'$ is a residue class $\bmod \, q$ coprime to its modulus. Thus, we have
\begin{multline}
  S_{\epsilon} (d_1, \dots , d_{k-1}, e_1, \dots , e_{k-1})   =  \frac{1}{\varphi (q')}   \sum_{\substack{ A_k x + B_k < m \leq 2A_k x  + B_k \\ (m,q')=1 }} \Omega^\flat (m)  \lambda_{F_{k}} ( m )  \lambda_{G_{k}}  ( m )  \mathbf{1}_{\text{lpf} (m) > x^{\epsilon}}\, \\
+ \Delta \left(  \Omega^\flat \lambda_{F_{k}}   \lambda_{G_{k}}   \mathbf{1}_{\text{lpf} (\cdot )> x^{\epsilon}}
\mathbf{1}_{[A_kx + B_k ,2A_kx + B_k]}; a' \bmod q' \right).
\end{multline}
We split
\begin{equation}
\sum_p S_{\epsilon} = S_1 - S_2 + S_3,
\end{equation}
where
\begin{equation}\label{Sumy_S}
\begin{split}
S_1 (d_1, \dots , d_{k-1}, e_1, \dots , e_{k-1})  &= \frac{1}{\varphi (q')} \sum_p \Upsilon (\log_x p) \sum_{\substack{ A_k x + B_k < m \leq 2A_k x  + B_k \\ p|m }}  \lambda_{F_{k}} ( m )  \lambda_{G_{k}}  ( m )  \mathbf{1}_{\text{lpf} (m) > x^{\epsilon}} , \\  
S_2 (d_1, \dots , d_{k-1}, e_1, \dots , e_{k-1}) &= \frac{1}{\varphi (q')}   \sum_{\substack{ A_k x + B_k < m \leq 2A_k x  + B_k \\ (m,q')>1 }} \Omega^\flat (m)  \lambda_{F_{k}} ( m )  \lambda_{G_{k}}  ( m )  \mathbf{1}_{\text{lpf} (m) > x^{\epsilon}} , \\
S_3(d_1, \dots , d_{k-1}, e_1, \dots , e_{k-1})  &= \Delta \left(  \Omega^\flat  \lambda_{F_{k}}   \lambda_{G_{k}}   \mathbf{1}_{\text{lpf} (\cdot )> x^{\epsilon}} \mathbf{1}_{[A_kx + B_k ,2A_kx + B_k]}; a' \bmod q' \right).
\end{split}
\end{equation}
For $j \in \{ 1,2,3 \}$ we put
\begin{equation}
\Sigma_j = 
 \sum_{\substack{ d_1, \dots , d_{k-1} \\ e_1, \dots , e_{k-1} }} \left( \prod_{i=1}^{k-1} \mu (d_i) \mu (e_i) F_i ( \log_x d_i) G_i (\log_x e_i)  \right) 
 S_j (d_1, \dots , d_{k-1}, e_1, \dots , e_{k-1}).
\end{equation}
Therefore, it suffices to derive the main term estimate
\begin{equation}\label{Sigma1}
\Sigma_1 = (c_\epsilon + o(1)) B^{-k} \frac{x}{ W }, \\
\end{equation}
the `correction' error term estimate
\begin{equation}\label{Sigma2}
\Sigma_2 \ll x^{1-\epsilon+o(1)}, \\
\end{equation}
and the `GEH-type' error term estimate
\begin{equation}\label{Sigma3}
\Sigma_3 \ll x \log^{-A} x
\end{equation}
for any fixed $A>0$. 

Let us begin with (\ref{Sigma2}). We observe that since $\text{lpf}(m)>x^\epsilon$, there exists a prime $x^\epsilon < p \leq x$ dividing both $m$ and one of $d_1, e_1, \dots ,d_{k-1}, e_{k-1}$ (if $k=1$, then $\Sigma_2$ vanishes; we also claim that $\epsilon$ tends to 0 slowly enough to ensure that $D_0 < x^\epsilon$). Thus, we may safely assume that $p|d_1$, for the remaining $2k-3$ cases are analogous. Hence, we get
\begin{equation}
\Sigma_2 \ll x^{o(1)} \sum_{x^\epsilon < p \leq x}  \sum_{\substack{ d_1, \dots , d_{k-1} \leq x \\ e_1, \dots , e_{k-1} \leq x \\ p|d_1}} 
\prod_{i=1}^{k-1} \frac{1}{\varphi( [d_i,e_i] ) } ~  \sum_{\substack{ n \ll x \\ p|n}}  1 \ll x^{1+o(1)}  \sum_{x^\epsilon < p \leq x} \frac{1}{p^2} \ll x^{1-\epsilon + o(1)}.
\end{equation}

To deal with (\ref{Sigma3}) we just repeat the reasoning from \cite[Subsection `The generalized Elliott-Halberstam case', Eq (62)]{Polymath8} combined with $\Omega^\flat (m) = O(1/\epsilon)$.

Let us move to (\ref{Sigma1}). We have 
\[ \varphi (q') = A_k \varphi \left( W \prod_{i=1}^{k-1} [d_i,e_i] \right),\]
 so again by \cite[Lemma 2.6]{Lewulis} (or \cite[Lemma 4.1]{Polymath8} for an even more direct application) we get
\begin{equation}
 \sideset{}{'}\sum_{\substack{ d_1, \dots , d_{k-1} \\ e_1, \dots , e_{k-1} }} \frac{ \prod_{i=1}^{k-1} \mu (d_i) \mu (e_i) F_i ( \log_x d_i) G_i (\log_x e_i)   }{\varphi \left( q' \right) } =  \frac{A_k^{-1}}{\varphi (W)} (c' + o(1)) B^{1-k}, 
\end{equation}
where 
\[ c' :=  \prod_{i=1}^{k-1} \int_0^1 F'_i (t) G'_i (t) \, dt. \]
By (\ref{Sumy_S}) it suffices to show that 
\begin{equation}\label{Cebis}
\sum_p \, \Upsilon (\log_x p)  \sum_{\substack{ A_k x + B_k < m \leq 2A_k x  + B_k \\ p|m }} \lambda_{F_{k}} ( m )  \lambda_{G_{k}}  ( m )  \mathbf{1}_{\text{lpf} (m) > x^{\epsilon}} = \left( c_\epsilon''  + o(1) \right) \frac{A_k  x}{\log x} ,
\end{equation}
where $c_\epsilon''$ satisfies 
\begin{equation}
\lim_{\epsilon \rightarrow 0} c_\epsilon'' =  \Upsilon(1) \, F_k (0) \,   G_k (0) ~+~ \int_0^1 \frac{\Upsilon (y)}{y}  \int_0^{1-y}  \partial_{y} F'_k ( t)
\, \partial_{y} G'_k  ( t  )  \, dt \, dy.
\end{equation}
We simplify the restriction $ A_k x + B_k < m \leq 2A_k x  + B_k $ into $m \sim A_kx$ at the cost of introducing to the left-hand side of (\ref{Cebis}) an error term of size not greater than $x^{o(1)}$. We factorize $m = p_1 \cdots p_r p$ for some $x^\epsilon \leq p_1 \leq \dots \leq p_r \leq 2A_kx  $, $p \geq x^\epsilon$, and $0 \leq r \leq \frac{1}{\epsilon} $. The contribution of those $m$ having repeated prime factors is readily $\ll x^{1-\epsilon}$, so we can safely assume that $m$ is square-free. In such a case, we get
\begin{equation}
 \lambda_{F_{k}} ( m )  = (-1)^r    \partial_{\, \log p_1} \dots \partial_{\, \log p_r}  ( \partial_{\, \log p } F_k (0) ) 
\end{equation}
and an analogous equation for $\lambda_{G_k}(m)$. Therefore, the left-hand side of (\ref{Cebis}) equals
\begin{equation}\label{suma_lambdy}
 \sum_{0 \leq r \leq \frac{1}{\epsilon } } \, \sum_p \Upsilon (\log_x p)  \sum_{\substack{ x^\epsilon < p_1 < \dots < p_r \\ p_1\dots p_r p \, \sim A_k x}}
    \partial_{\, \log p_1} \dots \partial_{\, \log p_r} ( \partial_{\, \log p } F_k (0) ) \, \cdot \,
  \partial_{\, \log p_1} \dots \partial_{\, \log p_r} ( \partial_{\, \log p } G_k (0) ) .
\end{equation}
Note that for the index $r=0$ the summand above equals
\begin{equation}\label{req0}
( \Upsilon (1) + o(1))  \sum_{p \sim A_k x} \,   F_k (0) \,   G_k (0).
\end{equation}
We apply Lemma \ref{sumynacalki} to (\ref{suma_lambdy}--\ref{req0}) and obtain an asymptotic (\ref{Cebis}) with
\begin{equation}\label{wynik_przed_przejsciem}
\begin{split}
c_\epsilon'' = \sum_{1\leq r \leq \frac{1}{\epsilon}}  \int_0^1 \Upsilon (y) \int_{\substack{  \phantom{2} \\ t_1+\dots+t_r=1-y \\ \epsilon < t_1 < \dots < t_r}}
 \partial_{t_1} \dots \partial_{ t_r} ( \partial_{ y } F_k (0) ) \cdot \partial_{t_1} \dots \partial_{ t_r} ( \partial_{ y } G_k (0) )
  \frac{ dy \,dt_1 \dots dt_{r-1}}{y \, t_1\cdots t_r} \\
+ ~ \Upsilon (1) \, F_k (0) \, G_k (0)  .
  \end{split}
\end{equation}
The first part of Lemma \ref{Almost primality} gives us $c_\epsilon'' \ll 1$ when $\epsilon \rightarrow 0^+$. Now, consider any  sequence of positive numbers $( \epsilon_1, \epsilon_2, \dots)$ satisfying $\epsilon_n \rightarrow 0$  as $n \rightarrow \infty$. In view of (\ref{Cebis}) and the last part of Lemma \ref{Almost primality}, we conclude that $\left( c_{\epsilon_1}'', c_{\epsilon_2}'' \dots \right)$ forms a Cauchy sequence, and hence it has a limit. Thus, by dominated convergence theorem it suffices to establish for each $y \in [0,1]$ the following equality
\begin{multline}
 \sum_{r \geq 1} \int_{\substack{ \phantom{2} \\ t_1+\dots+t_r=1-y \\ 0 < t_1 < \dots < t_r }}
 \partial_{t_1} \dots \partial_{ t_r} ( \partial_{ y } F_k (0) ) \cdot \partial_{t_1} \dots \partial_{ t_r} ( \partial_{ y } G_k (0) )
  \frac{ dt_1 \dots dt_{r-1}}{ t_1\cdots t_r}   \\
  =   \int_0^{1-y}   \partial_{y} F_k' ( t) \, \partial_{y} G_k' ( t)   \, dt.
  \end{multline}
By depolarization argument it suffices to show that for each $y \in [0,1] $, we have
\begin{equation}\label{po_depolaryzacji}
 \sum_{ r \geq 1 }  \int_{\substack{ \phantom{2} \\ t_1+\dots+t_r=1-y \\ 0 < t_1 < \dots < t_r }}
\left| \partial_{t_1} \dots \partial_{ t_r} ( \partial_{ y } F (0) ) \right|^2
  \frac{ \,dt_1 \dots dt_{r-1}}{ t_1\cdots t_r} =
   \int_0^{1-y} \left| \partial_{y} F' ( t) \right|^2 \, dt 
\end{equation}
for any smooth $F \colon [0,\infty) \rightarrow \mathbf{R}$. For the sake of clarity, we relabel $\partial_{ y } F (x) $ as $H(x)$. We substitute 
$u := t/(1-y)$ and $u_i := t_i/(1-y)$ for all possible choices of $i$. With these settings (\ref{po_depolaryzacji}) is equivalent to 
 \begin{equation}
\sum_{ r \geq 1 }  \int_{\substack{ \phantom{2} \\ u_1 + \dots + u_r =1 \\ 0 < u_1 < \dots < u_r }}
\left| \partial_{(1-y)u_1} \dots \partial_{ (1-y)u_r} H(0) \right|^2
  \frac{ \,du_1 \dots du_{r-1}}{ u_1\cdots u_r} =
   (1-y)^2 \int_0^{1} \left| H'(u(1-y)) \right|^2 \, du .
 \end{equation}
Note that one of the $(1-y)$ appeared from transforming $t_r \mapsto u_r$. Put $\widetilde{H}(x):=H(x(1-y))$. We get 
\[ \partial_{(1-y)u_1} \dots \partial_{ (1-y)u_r} H(0) =  \partial_{u_1} \dots \partial_{ u_r} \widetilde{H}(0),\]
 and $\widetilde{H}'(x)=(1-y)H'(x(1-y))$ by the chain rule. Thus, it suffices to show that
 \begin{equation}
 \sum_{ r \geq 1 }  \int_{\substack{ \phantom{2} \\ u_1 + \dots + u_r =1 \\ 0 < u_1 < \dots < u_r }}
\left| \partial_{u_1} \dots \partial_{ u_r} \widetilde{H}(0) \right|^2
  \frac{ \,du_1 \dots du_{r-1}}{ u_1\cdots u_r} =
 \int_0^{1} \left| \widetilde{H}'(u) \right|^2 \, du .
 \end{equation}
 To this end, we apply the key combinatorial identity \cite[(67)]{Polymath8}. 
 \end{proof}


\section{Proof of Theorem \ref{eps_simplex_sieving} }\label{proof_theo}

\begin{proof}
Let $k, m, \varepsilon, \theta, \ell$ be as in Theorem \ref{eps_simplex_sieving}. Let us assume that we have a non-zero square integrable function $F \colon [0,+\infty )^k \rightarrow \mathbf{R}$ supported on $(1+\varepsilon ) \cdot \mathcal{R}_k' \cap \, \eta \cdot \mathcal{R}_k$ and satisfying
\begin{equation}
 \frac{ \sum_{i=1}^k \left( Q_{i,\varepsilon}  (F)   - \theta(\ell -1 ) J_{i,\varepsilon} (F) \right) }{I(F)} +\ell k  < m.
\end{equation}
Now, we perform an analogous sequence of simplifications as in \cite[(72--84)]{Polymath8} and eventually arrive at a non-zero  smooth function $f \colon \mathbf{R}^k \rightarrow \mathbf{R}$ being the linear combination of tensor products -- namely
\begin{equation}
f (t_1, \dots , t_k) = \sum_{j=1}^J c_j f_{1,j} (t_1) \cdots f_{k,j} (t_j)
\end{equation} 
with $J$, $c_j$, $f_{i,j}$ fixed, for which all the components $ f_{1,j} (t_1) , \dots , f_{k,j} (t_k)$ are supported on the region
\begin{multline}
\left\{  (t_1, \dots , t_k) \in \mathbf{R}^k \colon \sum_{i=1}^k \max \left(  t_i , \delta \right) \leq \theta  \eta - \delta  \right\} \\ \cap 
\left\{  (t_1, \dots , t_k) \in \mathbf{R}^k \colon \forall_{1 \leq i_0 \leq k}
\sum_{\substack{1 \leq i \leq k \\ i \not= i_0}} \max \left(  t_i , \delta \right) \leq (1+ \varepsilon) \theta  - \delta  \right\}
\end{multline}
for some sufficently small $\delta >0$ -- that obeys
\begin{equation}\label{criterion_satisfied}
 \frac{ \sum_{i=1}^k \left( \widetilde{Q}_{i,\varepsilon}  (f)   - (\ell -1 ) \widetilde{J}_{i,\varepsilon} (f) \right) }{ \widetilde{I}(f)} +\ell k  < m,
\end{equation}
where 
\begin{align}
\widetilde{I} (f) :=& \int \limits_{[0, +\infty)^k} \left| \frac{\partial^k}{\partial t_1 \dots \partial t_k } 
f(t_1, \dots, t_k) \right|^2 dt_1 \dots   dt_k, \\
\widetilde{J}_{i,\varepsilon} (f) :=& \int \limits_{(1-\varepsilon)\theta \cdot \mathcal{R}_{k-1} } 
\left| \frac{\partial^{k-1}}{\partial t_1 \dots \partial t_{i-1} \partial t_{i+1} \dots  \partial t_k } 
f(t_1, \dots, t_{i-1}, 0 , t_{i+1} , \dots  ,  t_k) \right|^2 dt_1 \dots   dt_{i-1} dt_{i+1} \dots   dt_{k}  , \nonumber  \\
\widetilde{Q}_{i,\varepsilon} (f) :=& \int \limits_0^1 \frac{1- \ell y}{y}  \int \limits_{ {\Psi} (y) \cdot \mathcal{R}_{k-1} } \left( \int \limits_0^{1-y} \left| \partial_y^{(i)} \frac{\partial^k}{\partial t_1 \dots \partial t_k } 
f(t_1, \dots, t_k) \right|^2 dt_i \right)  dt_1 \dots   dt_{i-1} \, dt_{i+1} \dots   dt_k  \, dy, \nonumber
\end{align}
with $\Psi \colon [0,+\infty ) \rightarrow \mathbf{R}$ being a function given as
\begin{equation}
\Psi (y) := 
\begin{cases}
1+ \varepsilon, & \text{for } y \in \left[ 0 ,  \frac{1}{\ell } \right) , \\
1 - \varepsilon, & \text{for } y \in \left[ \frac{1}{\ell }  , 1 \right] , \\
0, & \text{otherwise.}
\end{cases}
\end{equation}
We construct a non-negative sieve weight $\nu \colon \mathbf{N} \rightarrow \mathbf{Z}$ by the formula
\begin{equation}\label{tensoring}
\nu (n) := \left( \sum_{j=1}^J c_j   \lambda_{f_{1,j}} ( L_1 (n)) \cdots \lambda_{f_{k,j}} ( L_k (n)) \right)^2.
\end{equation}
Notice that if $\varepsilon >0$, then for any $1 \leq j,j' \leq J$ we have
\begin{equation}
\sum_{i=1}^k (S(f_{i,j}) + S(f_{i,j'}))  < 2 \theta \eta < 1
\end{equation}
from the $2 \theta \eta + \frac{1}{\ell} \leq 1$ assertion. On the flip side, if $\varepsilon = 0$, then $\text{supp} (F) \subset \mathcal{R}_k'$ and consequently for every $1 \leq i_0 \leq k$ we have
\begin{equation}
\sum_{\substack{ 1 \leq i \leq k \\ i \not= i_0}} (S(f_{i,j}) + S(f_{i,j'}))  < 2 \theta .
\end{equation}
Applying results from \cite[Subsection `Proof of Theorem 3.12']{Polymath8}, we get
\begin{equation}
 \sum_{ \substack{ n \sim x \\ n \equiv b \bmod W }}   \nu (n) = \left( \alpha + o(1) \right) B^{-k} \frac{x}{W},
\end{equation}
where
\[ \alpha = \widetilde{I} (f). \]

Now, let us consider the sum
\begin{equation}
 \sum_{ \substack{ n \sim x \\ n \equiv b \bmod W \\ \mathcal{P}(n) \textup{ sq-free} }}     \nu (n)
 \sum_{p | L_k (n) } \left( 1 - \ell \log_x p \right)  .
\end{equation}
We can expand the sum above as a linear combination of expressions
\begin{equation}\label{rozpad}
 \sum_{ \substack{ n \sim x \\ n \equiv b \bmod W \\ \mathcal{P}(n) \textup{ sq-free} }}   
 \sum_{p | L_k (n) } \left( 1 - \ell \log_x p \right)   \prod_{i=1}^k \lambda_{f_{i,j}} ( L_i (n)) \lambda_{f_{i,j'}} ( L_i (n))
\end{equation} 
for various $1 \leq j,j' \leq J$. We seek for the upper bound of the sum (\ref{rozpad}). We can achieve this goal by applying Proposition \ref{Powerful_Proposition}. We also observe that the first part of this result should be more effective for smaller values of $p$, and the second part for larger values of $p$. Therefore, we perform a decomposition of the expression (\ref{rozpad}) as follows:
\begin{equation}
 \sum_{ \substack{ n \sim x \\ n \equiv b \bmod W \\ \mathcal{P}(n) \textup{ sq-free} }}   
 \sum_{p | L_k (n) } = 
  \sum_{ \substack{ n \sim x \\ n \equiv b \bmod W \\ \mathcal{P}(n) \textup{ sq-free} }}   
\left( \sum_{\substack{p | L_k (n) \\ p \leq x^{1/\ell} }} + \sum_{\substack{p | L_k (n) \\ p > x^{1/\ell} }} \right).
\end{equation}
For the $p \leq x^\ell$ sum we apply the trivial case of Proposition \ref{Powerful_Proposition} with $\vartheta_0 = 1/\ell$ and 
\[ \Upsilon (y) = (1- \ell y)\mathbf{1}_{y\leq 1/\ell }. \]
Under these assumptions we have 
\begin{align}\label{nasze_warunki}
\sum_{i=1}^k \left( S(f_{i,j}) + S(f_{i,j'}) \right) &< 2\theta (1+ \varepsilon) \leq 1 - \frac{1}{\ell},   \\
S ( \Upsilon ) &\leq  \frac{1}{\ell} ,
\end{align}
so the necessary hypotheses from the `trivial case' of Proposition \ref{Powerful_Proposition} are indeed satisfied. Observe that under $\varepsilon =0$ the inequality (\ref{nasze_warunki}) satisfies the second case of Proposition \ref{Powerful_Proposition}, so in these circumstances we do not have to rely on the constraint (\ref{wiezy}) any longer. Thus, we get
\begin{equation}\label{wynik_epsilon_beta1}
 \sum_{ \substack{ n \sim x \\ n \equiv b \bmod W \\ \mathcal{P}(n) \textup{ sq-free} }}     \nu (n)
 \sum_{\substack{p | L_i (n) \\ p \leq x^{1/\ell} }} \left( 1 - \ell \log_x p \right)  = \left( \beta_k^{(1)} + \, o(1) \right) B^{-k} \frac{x}{W},
\end{equation} 
where
\begin{equation}
\begin{split}
\beta_k^{(1)} =  
 \sum_{j,j'=1}^J c_j c_{j'}  \left(   \int    \limits_0^1 \frac{\Upsilon ( y)}{y}  \int  \limits_0^{1-y}  \partial_{y} f_{k,j}'(t_k) 
\, \partial_{y} f_{k,j'}'(t_k)  \, dt_{k} \, dy \right) 
 \prod_{i=1}^{k-1} \left(  \int \limits_0^1 f_{k,j}'(t_i) \, f_{k,j'}'(t_i) \, dt_i \right) .
\end{split}
\end{equation} 
From (\ref{tensoring}) we see that $\beta_k^{(1)}$ factorizes as 
\begin{equation}
\beta_k^{(1)} = 
 \int \limits_0^{1/\ell} \frac{1- \ell y}{y}  \int \limits_{ (1+\varepsilon) \theta \cdot \mathcal{R}^{k-1}  }  \int \limits_0^{1-y} \left| \partial_y^{(k)} \frac{\partial^k}{\partial t_1 \dots \partial t_k } 
f(t_1, \dots, t_k) \right|^2 dt_i \, dt_1 \dots    dt_{k-1}  \, dy.
\end{equation}

Now we deal with the $p> x^\ell$ case. We apply the $GEH$ case of Proposition \ref{Powerful_Proposition} with $\vartheta = 1/2$ and 
\[ \Upsilon (y) = (1- \ell y)\mathbf{1}_{y> 1/\ell }. \]
We decompose $\{1, \dots , J\}$ into $\mathcal{J}_1 \cup \mathcal{J}_2$, where $\mathcal{J}_1$ consists of those indices $j  \in \{1, \dots , J\}$ satisfying
\begin{equation}\label{epsilon_trick_support}
\sum_{i=1}^{k-1} S(f_{i,j})  < (1 - \varepsilon ) \theta,
\end{equation}
and $\mathcal{J}_2$ is the complement. As in \cite{Polymath8} we apply the elementary inequality
\[ (x_1 + x_2)^2  \geq  (x_1 + 2x_2)x_1 \]
to obtain the pointwise lower bound
\begin{equation}\label{lower_bound_epsilon}
\begin{split}
\nu (n) \geq \left( \left( \sum_{j \in \mathcal{J}_1} + ~ 2 \sum_{j \in \mathcal{J}_2}  \right)
c_j \lambda_{f_{1,j}} (L_1 (n)) \cdots \lambda_{f_{k,j}} (L_k (n)) \right)
\left(  \sum_{j' \in \mathcal{J}_1}
c_{j'} \lambda_{f_{1,j'}} (L_1 (n)) \cdots \lambda_{f_{k,j'}} (L_k (n)) \right).
\end{split}
\end{equation}
Therefore, if $j \in \mathcal{J}_1 \cup \mathcal{J}_2$ and $ j' \in \mathcal{J}_1$, then from (\ref{epsilon_trick_support}) one has 
\[ \sum_{i=1}^{k-1} \left( S(f_{i,j}) + S(f_{i,j'}) \right) < 2  \theta,  \]
so the hypothesis from the `Generalised Elliott--Halberstam' case of Proposition \ref{Powerful_Proposition} is indeed satisfied. Thus, by Proposition \ref{Powerful_Proposition} and (\ref{lower_bound_epsilon}) we get
\begin{equation}
 \sum_{ \substack{ n \sim x \\ n \equiv b \bmod W \\ \mathcal{P}(n) \textup{ sq-free} }}     \nu (n)
 \sum_{\substack{p | L_i (n) \\ p > x^{1/\ell} }} \left( 1 - \ell \log_x p \right)  \leq \left( \beta_k^{(2)} + \, o(1) \right) B^{-k} \frac{x}{W},
\end{equation} 
where
\begin{equation}
\begin{split}
\beta_k^{(2)} =    
 \left( \sum_{j \in \mathcal{J}_1} + ~ 2 \sum_{j \in \mathcal{J}_2}  \right)  \sum_{j' \in \mathcal{J}_1}
 c_j c_{j'}  \left( \Upsilon (1)  \, f_{k,j}(0) \, f_{k,j'}(0) +   \int    \limits_0^1 \frac{\Upsilon ( y)}{y}  \int  \limits_0^{1-y}  \partial_{y} f_{k,j}'(t_k) 
\, \partial_{y} f_{k,j'}'(t_k)  \, dt_{k} \, dy  \right) \\ 
\times  \, \prod_{i=1}^{k-1} \left(  \int \limits_0^1 f_{k,j}'(t_i) \, f_{k,j'}'(t_i) \, dt_i \right) .
\end{split}
\end{equation} 
For $s = 1,2$ let us define
\[f_s (t_1, \dots , t_k ):= \sum_{j \in \mathcal{J}_s} c_j f_{1,j} (t_1) \cdots f_{k,j} (t_k). \]
From (\ref{tensoring}) we observe that $\beta_k^{(2)}$ can be factorized as
\begin{equation}
\beta_k^{(2)} = \beta_k^{(2,1)}  + \beta_k^{(2,2)} ,
\end{equation}
where
\begin{multline*}
\beta_k^{(2,1)} := 
 \int \limits_{1/\ell}^1 \frac{1- \ell y}{y}  \int  \limits_{ (1 - \varepsilon)\theta \cdot \mathcal{R}^{k-1}  }   \int \limits_0^{1-y} \left( \partial_y^{(k)} \frac{\partial^k}{\partial t_1 \dots \partial t_k } 
f_1(t_1, \dots, t_k) +  2 \partial_y^{(k)} \frac{\partial^k}{\partial t_1 \dots \partial t_k } f_2(t_1, \dots, t_k) \right)  \\
\times \partial_y^{(k)} \frac{\partial^k}{\partial t_1 \dots \partial t_k } f_1(t_1, \dots, t_k)  \, dt_i \, dt_1 \dots    dt_{k-1}  \, dy  
\end{multline*}
and
\begin{multline*}
\beta_k^{(2,2)} := (1 - \ell ) \int \limits_{(1-\varepsilon)\theta \cdot \mathcal{R}_{k-1} } 
\left( \frac{\partial^{k-1}}{\partial t_1   \dots  \partial t_{k-1} }  
f_1(t_1, \dots, t_{k-1}, 0 ) + 2 \frac{\partial^{k-1}}{\partial t_1   \dots  \partial t_{k-1} } f_2(t_1, \dots, t_{k-1}, 0 ) \right)  \\
\times \frac{\partial^{k-1}}{\partial t_1   \dots  \partial t_{k-1} }f_1(t_1, \dots, t_{k-1}, 0 ) \, dt_1 \dots    dt_{k-1} .
\end{multline*}
Let $\delta_1 > 0$ be a sufficiently small fixed quantity. By a smooth
partitioning, we may assume without loss of generality that all of the $f_{i,j}$ are supported on intervals of length at most $\delta_1$, while keeping the sum
\[ \sum_{j=1}^J |c_j| | f_{1,j} (t_1)| \cdots |f_{k,j} (t_k)|  \]
bounded uniformly in $t_1, \dots , t_k$ and in $\delta_1$. Therefore, the supports of $f_1$ and $f_2$ overlap only on some set of measure at most $O (\delta_1 )$. Hence, we conclude that 
\begin{equation}
\beta_k := \beta_k^{(1)} + \beta_k^{(2)} = \, \widetilde{J}_{k, \varepsilon} (f) + \widetilde{Q}_{k, \varepsilon} (f)  + O(\delta_1),
\end{equation}
which implies
\begin{equation}\label{wynik_102}
 \sum_{ \substack{ n \sim x \\ n \equiv b \bmod W \\ \mathcal{P}(n) \textup{ sq-free} }}     \nu (n)
 \sum_{p | L_k (n) } \left( 1 - \ell \log_x p \right)  \leq \left( \beta_k + o(1) \right) B^{-k} \frac{x}{W}.
 \end{equation}
 A similar argument provides results analogous to (\ref{wynik_102}) for all remaining indices $1 \leq i \leq k-1$. If we set $\delta_1$ to be small enough, then the claim $DHL_\Omega [k; \varrho_k]$ follows from Lemma \ref{criterion} and (\ref{criterion_satisfied}). We also note that if $\varepsilon =0$, then (\ref{wynik_102}) becomes an equality, because in this case we have $\mathcal{J}_2 = \emptyset $.
\end{proof}


\section{Solving variational problems}

In this Section we focus on applying Theorems \ref{st_simplex_sieving}, \ref{ext_simplex_sieving}, and \ref{eps_simplex_sieving} to prove Theorem \ref{MAIN}. 

\subsection{Proof of Theorem \ref{1dim_sieving}}\label{PT1D}

\begin{proof}
This is a direct application of Theorem \ref{st_simplex_sieving}. We choose $F(t_1, \dots , t_k) = \bar{f} (t_1+\dots+t_k)$ for a function $\bar{f} \colon [0,+\infty) \rightarrow \mathbf{R}$ defined as
\begin{equation}
\bar{f} (x) := 
\begin{cases}
f(x), & \text{for } x \in [0,1] , \\
0, & \text{otherwise.}
\end{cases}
\end{equation}
We also set $\ell=1$, so the contribution from $J_i (F)$ vanishes for each possible choice of index $i$. 

First, we calculate $I(F)$. We substitute $t_1+\dots+t_k \mapsto t$ and leave $t_j$ the same for $j=2, \dots , k$. We get
\begin{equation}\label{calculating_if}
I(F)  =  \int \limits_0^1 f(t)^2 \left( \int \limits_{t \cdot \mathcal{R}_{k-1} } dt_2  \dots  dt_k \right) dt ~=~ \frac{1}{(k-1)!}\, \int \limits_0^1 f(t)^2 \, t^{k-1} dt ~=~ \bar{I} (f). 
\end{equation}

Let us move on to the $Q_i(F)$ integral. For the sake of convenience let us choose $i=k$. By the same substitution as before we arrive at
\begin{equation}\label{1dim_poczatek}
Q_k (F) = \int \limits_0^\frac{1}{\theta} \frac{1 - \theta y}{y} \int \limits_0^1 \left( \bar{f} (t) - \bar{f} (t+y) \right)^2
 \int \limits_{t \cdot \mathcal{R}_{k-1} } \mathbf{1}_{t_k \leq \frac{1}{\theta} - y} \, dt_2  \dots  dt_k \, dt \, dy.
\end{equation}
We wish to replace $\bar{f}$ with $f$ and discard the indicator function. The latter can be simply performed by calculating the inner integral. Note that it may be geometrically intepreted as a volume of a `bitten' simplex. We define
\[ H_{y,t} := \left\{ ( t_2, \dots , t_k ) \in \mathbf{R}^{k-1}  \colon t_2 + \dots + t_k \leq t ~~\text{and}~~ t_k > 1/\theta  - y 
\right\}. \]
Observe that $H_{y,t}$ is just a translated simplex $(t-1/\theta + y) \cdot \mathcal{R}_{k-1}$ for $1/\theta - t < y \leq 1/\theta$ and an empty set for $y \leq 1/\theta - t $ . Thus, we obtain
\begin{multline}
\frac{1}{(k-1)!}\int \limits_{t \cdot \mathcal{R}_{k-1} } \mathbf{1}_{t_k \leq \frac{1}{\theta} - y} \, dt_2  \dots  dt_k   \\
=  \text{Vol} ( t \cdot \mathcal{R}_{k-1} ) -  \text{Vol} ( H_{y,t} ) =
 \begin{cases}
t^{k-1}, & \text{for } y \in [0,\frac{1}{\theta}-t] , \\
 t^{k-1} - (t - 1/\theta + y)^{k-1}, & \text{for } y \in (\frac{1}{\theta}-t,\frac{1}{\theta}].
\end{cases} 
  \end{multline}
For $0\leq y \leq1$ we also have
\begin{equation}
 \bar{f} (t) - \bar{f} (t+y)  = 
 \begin{cases}
f(t)-f(t+y), & \text{for } t \in [0,1-y] , \\
f(t), & \text{for } t \in (1-y,1],
\end{cases} 
\end{equation}
and simply $ \bar{f} (t) - \bar{f} (t+y)  = \bar{f} (t)$ for greater $y$. We decompose the domain of integration 
\[ D := \{ (y,t) \in \mathbf{R}^2  \colon 0 < t < 1 ~\text{and}~ 0 < y < 1/\theta  \}\]
into
\begin{equation}\label{1dim_koniec}
 D = D_1 \cup D_2 \cup D_3 \cup D_4 \cup D_5 \cup \left( \text{some set of Lebesgue measure 0} \right) , 
 \end{equation}
where
\begin{align*}
D_1 &:= \{ (y,t) \in \mathbf{R}^2  \colon 0 < y < 1 ~\text{and}~ 0 < t < 1-y   \}, \\
D_2 &:= \{ (y,t) \in \mathbf{R}^2  \colon  0 < y < 1 ~\text{and}~1-y < t < 1 \},\\
D_3 &:= \{ (y,t) \in \mathbf{R}^2  \colon 1 < y < 1/\theta-1 ~\text{and}~ 0 < t < 1  \},\\
D_4 &:= \{ (y,t) \in \mathbf{R}^2  \colon 1/\theta -1 < y < 1/\theta ~\text{and}~ 0 < t < 1/\theta -y  \}, \\
D_5 &:= \{ (y,t) \in \mathbf{R}^2  \colon 1/\theta - 1 < y < 1/\theta ~\text{and}~ 1/\theta - y < t < 1   \}.
\end{align*}
Therefore, from (\ref{1dim_poczatek}--\ref{1dim_koniec}) we get
\begin{multline}\label{QF_wynik_st}
Q_k (F) =   \iint   \limits_{D_1}   \frac{1 - \theta y}{y} \, ( f(t) - f(t+y))^2\,  t^{k-1} \, dt \, dy ~   \\
+ \iint   \limits_{D_2 \cup D_3 \cup D_4}   \frac{1 - \theta y}{y} \, f(t)^2\,  t^{k-1} \, dt \, dy ~+~
\iint   \limits_{D_5}   \frac{1 - \theta y}{y} \, f(t)^2\, \left(  t^{k-1} - \left(t+y- 1/\theta  \right)^{k-1} \right) \, dt \, dy.
\end{multline}
The same reasoning applies to $Q_i (F)$ for $i =1, \dots , k-1$.



\subsection{Collapse of Theorem \ref{ext_simplex_sieving} into one dimension and bounds for $\Omega_k^{\textmd{ext}}$}\label{bounds_omegaextk}

We wish to transform Theorem \ref{ext_simplex_sieving} into its one-dimensional analogue in a similar manner as we did in 
Subsection \ref{PT1D}. For the sake of convenience, let us assume in this subsection that $k \geq 3$. In the $k=2$ case, Theorem \ref{ext_simplex_sieving} can be applied directly without any intermediate simplifications -- it also does not provide anything beyond what is already known anyway, as presented in Table D. 
We take 
\[F(t_1, \dots , t_k) = f(t_1 + \dots + t_k) \mathbf{1}_{(t_1,\dots,t_k) \in \mathcal{R}_{k}' },\]
where $f \colon [0,+\infty ) \rightarrow \mathbf{R}$ is some locally square-integrable function. We also put $\ell=1$, so the contribution from $J_i (F)$ vanishes for each possible choice of index $i$. Let us begin with $I(F)$ integral. This time we substitute
\begin{equation}\label{substitution}
 \begin{cases}
t_1+\dots + t_k & \longmapsto ~~x, \\
t_1+\dots + t_{k-1} & \longmapsto ~~t, \\
t_1 & \longmapsto ~~t_1, \\
&\vdots \\
t_{k-2} & \longmapsto ~~t_{k-2}.
\end{cases}
\end{equation}
We also relabel $t_{k-1}$ as $s$. It is calculated in \cite[Subsubsection `Calculating J']{Lewulis} that
\begin{equation}\label{j_ext}
I(F) = \int \limits_{\mathcal{R}_{k}'} F(t_1, \dots , t_k)^2 \, dt_1 \dots dt_k ~=~
\frac{1}{(k-3)!} \int \limits_0^1 \int \limits_0^t \int \limits_t^{1+\frac{s}{k-1}} f(x)^2 \, (t-s)^{k-3} \, dx \, ds \, dt.  
\end{equation} 

Let us focus on the $Q_i(F)$ integral. Again, for the sake of convenience we choose $i=k$. We have
\begin{equation}
Q_k (F) = \int \limits_0^\frac{1}{\theta} \frac{1 - \theta y}{y}
 \int \limits_{ \mathcal{R}_{k-1} } 
 \left(  \int \limits_0^{\rho (t_1 , \dots , t_{k-1} )} \left( \partial_y \bar{f} (t_1+\dots+t_k)  \right)^2 \mathbf{1}_{t_k \leq \frac{1}{\theta} - y} \, dt_k \right) \, dt_1  \dots  dt_{k-1} \,  dy,
\end{equation}
where 
\[ \rho(t_1, \dots , t_k) := \sup \{ t_k \in \mathbf{R} \colon (t_1, \dots, t_k) \in \mathcal{R}_k' \}. \]
We observe that any permutation of the variables $t_1, \dots , t_{k-1}$ does not change the integrand. We also notice that if we consider an extra assertion $0 < t_1 < \dots  < t_{k-1}$, then
\[ \rho ( t_1, \dots , t_{k-1} ) = 1 - t_2 -  \dots - t_{k-1}. \]
Therefore, $Q_k (F)$ equals
\begin{equation}\label{expression112}
 (k-1)! \, \int \limits_0^\frac{1}{\theta} \frac{1 - \theta y}{y}
 \int \limits_{ \substack{ \mathcal{R}_{k-1}  \\ 0 < t_1 < \dots < t_{k-1} }} 
 \left(  \int \limits_0^{  1 - t_2 -  \dots - t_{k-1} } \left( \partial_y \bar{f} (t_1+\dots+t_k)  \right)^2 \mathbf{1}_{t_k \leq \frac{1}{\theta} - y} \, dt_k \right) \, dt_1  \dots  dt_{k-1} \,  dy,
\end{equation}
In order to calculate the inner integral, we perform the same substitution as described (\ref{substitution}). This way we obtain
\begin{multline}\label{expression113}
 \int \limits_{ \substack{ \mathcal{R}_{k-1}  \\ 0 < t_1 < \dots < t_{k-1} }} 
 \left(  \int \limits_0^{  1 - t_2 -  \dots - t_{k-1} } \left( \partial_y \bar{f} (t_1+\dots+t_k)  \right)^2 \mathbf{1}_{t_k \leq \frac{1}{\theta} - y} \, dt_k \right) \, dt_1  \dots  dt_{k-1}  \\
=~ \int \limits_0^1   \int \limits_{ \substack{ 0 < t_1 < \dots < t_{k-2} < t - \sum_{i=1}^{k-2} t_i }} 
 \left(  \int \limits_t^{  1 +t_1 } \left( \partial_y \bar{f} (x)  \right)^2 \mathbf{1}_{x-t \leq \frac{1}{\theta} - y} \, dx \right) \, dt_1  \dots  dt_{k-2} \, dt 
 \end{multline}
For the sake of clarity, we relabel $t_1$ as $s$. Thus, the expression from (\ref{expression113}) equals
\begin{equation}\label{calka_kombinatoryczna}
\int \limits_0^1  \int \limits_0^{\frac{t}{k-1}}
 \left(  \int \limits_t^{  1 + s } \left( \partial_y \bar{f} (x)  \right)^2 \mathbf{1}_{x-t \leq \frac{1}{\theta} - y} \, dx \right) 
 \left( \int \limits_{s}^{\frac{t-s}{k-2}} \int \limits_{t_2}^{\frac{t-s-t_2}{k-3}} \cdots \int \limits_{t_{k-3}}^{\frac{t-s-t_2-\dots -t_{k-3}}{2}} \, dt_{k-2}  \dots  dt_2 \right)
   \, ds \, dt.
\end{equation}
If $k=3$, then the inner integral simplifies to $1$. For $0\leq s \leq t$ let us define
\[ \mathscr{L} (k;t,s) := 
 \int \limits_{s}^{\frac{t-s}{k-2}} \int \limits_{t_2}^{\frac{t-s-t_2}{k-3}} \cdots \int \limits_{t_{k-3}}^{\frac{t-s-t_2-\dots -t_{k-3}}{2}} \, dt_{k-2}  \dots  dt_2 .\]
We apply the induction over $k$ to show that
\begin{equation}\label{claim_K}\mathscr{L} (k;t,s) = \frac{(t - (k-1)s)^{k-3}}{(k-2)!(k-3)!}.
\end{equation}
Our claim is obviously true for $k=3$. For every $k\geq 3$ we observe the identity
\begin{equation}\label{identity_K}
\mathscr{L} (k+1 ; t,s ) = \int \limits_s^{\frac{t-s}{k-1}} \mathscr{L} (k ; t-s , u ) \, du. 
\end{equation}
To finish the proof of the claim one has to put (\ref{claim_K}) into (\ref{identity_K}) and substitute 
\[t-s-(k-1)u \mapsto z.\]
Combining (\ref{expression112}--\ref{calka_kombinatoryczna}) with the claim discussed above we conclude that $Q_k (F) $ equals
\begin{equation}
\frac{(k-1)!}{(k-2)!(k-3)!} \, \int \limits_0^\frac{1}{\theta} \frac{1 - \theta y}{y}
\int \limits_0^1  \int \limits_0^{\frac{t}{k-1}}
 \left(  \int \limits_t^{  1 + s } \left( \partial_y \bar{f} (x)  \right)^2 \mathbf{1}_{x-t \leq \frac{1}{\theta} - y} \, dx \right) 
(t - (k-1)s)^{k-3}    \, ds \, dt \,  dy.
\end{equation}
Let us relabel $s$ as $s/(k-1)$ to simplify the expression above. We arrive at
\begin{align}\label{exp118}
Q_k (F) &= \frac{1}{(k-3)!} \, \int \limits_0^\frac{1}{\theta} \frac{1 - \theta y}{y}
\int \limits_0^1  \int \limits_0^t
 \left(  \int \limits_t^{  1 + \frac{s}{k-1} } \left( \partial_y \bar{f} (x)  \right)^2 \mathbf{1}_{x-t \leq \frac{1}{\theta} - y} \, dx \right) 
(t - s)^{k-3}    \, ds \, dt \,  dy  \nonumber \\
&=  \frac{1}{(k-3)!} \int \limits_E  \frac{1 - \theta y}{y}  \left( \partial_y \bar{f} (x)  \right)^2  (t - s)^{k-3} \, dx \, dt \, ds \,  dy , 
\end{align}
where 
\[ E := \left\{ (y,s,t,x) \in \mathbf{R}^4  \colon 0 < y < \frac{1}{\theta},~ 0 < t < 1,~ 0<s< t,~ t<x<1+\frac{s}{k-1},~ x-t< \frac{1}{\theta} - y   \right\}. \]
We wish to drop the bar from $\bar{f}$. Hence, we decompose
\[ E = E_1 \cup E_2, \]
where 
\begin{align*}
E_1 &:= \left\{ (y,s,t,x) \in E  \colon x+y \leq 1 + \frac{s}{k-1}   \right\}, \\
E_2 &:= \left\{ (y,s,t,x) \in E   \colon  x+y > 1 + \frac{s}{k-1} \right\}.\\
\end{align*}
From (\ref{exp118}) we have that $Q_k(F)$ equals $1/(k-3)!$ times
\begin{equation}
 \int \limits_{E_1}  \frac{1 - \theta y}{y}  \left( f(x) - f (x+y)  \right)^2  (t - s)^{k-3} \, dx \, dt \, ds \,  dy ~+~
  \int \limits_{E_2}  \frac{1 - \theta y}{y}  f (x)^2  (t - s)^{k-3} \, dx \, dt \, ds \,  dy.
\end{equation}
 Now, we would like to convert two integrals above into a finite sum of integrals with explicitely given limits, just like in (\ref{QF_wynik_st}). If we choose the order of integration
 \[ y \rightarrow x \rightarrow t \rightarrow s, \]
 then we get
 \begin{align}
  \int \limits_{E_1}  \boxtimes ~&=~ \int \limits_0^1 \int \limits_s^1 \int \limits_t^{1 + \frac{s}{k-1} } ~ \int \limits_{0}^{1 + \frac{s}{k-1} -x}  \boxtimes ~ dy \, dx \, dt \, ds, \label{j0_ext1} \\
    \int \limits_{E_2}  \boxtimes ~&=~  \int \limits_0^1  \int \limits_s^1  \int \limits_t^{1+\frac{s}{k-1} }  \int \limits_{1+\frac{s}{k-1}-x }^{\frac{1}{\theta}+t-x}   \boxtimes ~ dy \, dx \, dt \, ds, \label{j0_ext2}
 \end{align}
 where $\boxtimes$ denotes an arbitrary integrable function.  
  \begin{remark} 
   From the computational point of view, the variable $y$ should be integrated in the last order, because it engages a non-polynomial function. The author found the following order of integration as the most computationally convenient: 
 \[ x \rightarrow t \rightarrow s \rightarrow y. \] 
   Unfortunately, in this case there is no decompsition of $E$ similar to (\ref{1dim_koniec}) that is common for all possible choices of $k$ and $\theta$. In the $k=4, \, \theta=1/2$ case, which accordingly to Tables C and D is the only one, where we can expect a qualitative improvement over Theorem \ref{1dim_sieving}, we are able to convert the integral over $E$ into 15 integrals with explicitely given limits. Such a conversion is a straightforward operation (quite complicated to perform without a computer program, though). We do not present the precise shape of these integrals here. 
\end{remark}
Let us set $k=4$, $\theta=1/2$, and
\[ f(x) = 12+63x+100x^2. \]
Combining (\ref{j_ext}) with (\ref{j0_ext1}--\ref{j0_ext2}), and performing the calculations on a computer, we get
\begin{align*}
I(F) = ~&\frac{2977019}{51030} > 58.3386047422. \\
Q_k (F) =~ &\frac{132461570733345 \log \frac{5}{3} - 997242435 \log 3 - 49178701703144  }{4629441600} \\
& + \frac{6144554}{105} \log \frac{6}{5} - \frac{15996989}{280} \, \text{arcoth} \, 4 < 70.0214943902.
\end{align*}
This combined with Theorem \ref{ext_simplex_sieving} gives
\begin{equation}
\Omega_4^{\text{ext}} \left( \theta = \frac{1}{2} \right) < 8.80105,
\end{equation}
which proves the $k=4$ case of the conditional part of Theorem \ref{MAIN}.

\subsection{bounds for $\Omega_{k,\varepsilon}$} 
We apply Theorem \ref{eps_simplex_sieving} with $\eta = 1+ \varepsilon$ and some $\varepsilon$, $\ell$ satisfying
\[ 2 \theta (1 + \varepsilon ) + \frac{1}{\ell} = 1. \] 
 We choose $F(t_1, \dots , t_k) = \bar{f} (t_1+\dots+t_k)$ for a function $\bar{f} \colon [0,+\infty) \rightarrow \mathbf{R}$ satisfying
\begin{equation}
\bar{f} (x) := 
\begin{cases}
f(x), & \text{for } x \in [0,1+\varepsilon] , \\
0, & \text{otherwise.}
\end{cases}
\end{equation}
First, we calculate $I(F)$. We proceed just like in (\ref{calculating_if}) and get
\begin{equation}
I(F) \, = \ \int \limits_0^{1+\varepsilon} f(t)^2 \left( \, \int \limits_{ t \cdot \mathcal{R}_{k-1} } dt_2  \dots  dt_k \right) dt ~=~ \frac{1}{(k-1)!}\, \int \limits_0^{1+\varepsilon} f(t)^2 \, t^{k-1} \, dt. 
\end{equation}
Next, let us consider $J_{i,\varepsilon} (F)$. As before, let us put $i=k$. We have
\begin{equation}
J_{k,\varepsilon} (F) \, = \, \int \limits_{(1-\varepsilon ) \cdot \mathcal{R}_{k-1} } \left(  \int \limits_0^{1+\varepsilon - t_1 - \dots - t_{k-1} }  f (t_1 + \dots + t_k)  \, dt_k \right)^2   dt_1 \dots dt_{k-1}. 
\end{equation}
We perform the same substitution as in (\ref{calculating_if}). We get that $J_{k,\varepsilon} (F)$ equals
\begin{equation}
\begin{gathered}
\, \int \limits_0^{1-\varepsilon} \left(  \int \limits_0^{1+\varepsilon - t }  f (t + t_k)  \, dt_k \right)^2   \int \limits_{t \cdot \mathcal{R}_{k-1} } \, dt_1 \dots dt_{k-2} \, dt \, = \,  
\int \limits_0^{1-\varepsilon}  \left(  \int \limits_t^{1+\varepsilon  }  f (x)  \, dx \right)^2  \frac{t^{k-2}}{(k-2)!} \, dt.
\end{gathered}
\end{equation}
We perform analogous calculations for $i=1, \dots , k-1$. 

Let us move to $Q_{i,\varepsilon} (F)$. Put
\begin{equation}
 \begin{cases}
t_1+\dots + t_{k-1} & \longmapsto ~~t, \\
t_2 & \longmapsto ~~t_2, \\
&\vdots \\
t_k & \longmapsto ~~t_k.
\end{cases}
\end{equation} 
and split
\begin{equation}
Q_{k,\varepsilon}(F) = Q_{(1)}(f) + Q_{(2)}(f),
\end{equation}
where
\begin{align}
 Q_{(1)} (f) &:= \frac{1}{(k-2)!}  \int \limits_0^\frac{1}{\ell \theta}  \frac{1- \ell \theta y}{y}   \int \limits_0^{1+\varepsilon} \left(  \, \int \limits_0^{\frac{1}{\theta } - y} \left(  \bar{f} (t + t_k) - \bar{f} (t + t_k+y) \right)^2 \, dt_k \right) 
t^{k-2} \, dt \, dy, \nonumber \\
Q_{(2)} (f) &:= \frac{1}{(k-2)!}  \int \limits_{\frac{1}{\ell \theta}}^{\frac{1}{\theta}}   \frac{1- \ell \theta y}{y}   \int \limits_0^{1-\varepsilon} \left(  \, \int \limits_0^{\frac{1}{\theta } - y} \left(  \bar{f} (t + t_k) - \bar{f} (t + t_k+y) \right)^2 \, dt_k \right) 
t^{k-2} \, dt \, dy,
\end{align}
Therefore, we put $t_k+t \mapsto x$ and decompose
\begin{multline}
(k-2)!  \left( Q_{(1)} (f) + Q_{(2)}(f) \right) = \\[1ex]
\int \limits_{H_1 \cup H_3}  \frac{1- \ell \theta y}{y}  f(x)^2  \,
t^{k-2} \, dx \, dt \, dy \, + \,
 \int \limits_{H_2 \cup H_4}  \frac{1- \ell \theta y}{y}  \left(  f (x) - f (x+y) \right)^2  
t^{k-2} \, dx \, dt \, dy,  
\end{multline}
where
\[ H := \{ (y,t,x) \in \mathbf{R}^3 \colon 0<y<1/ \theta ,~ 0<t<x<1+\varepsilon ,~ x-t<1/\theta - y   \},\]
and
\begin{align*}
H_1 &:= \{ (y,t,x) \in H  \colon 0<y \leq 1/ (\ell \theta)  ~\text{and}~  x+y < 1+\varepsilon   \}, \\
H_2 &:= \{ (y,t,x) \in H  \colon 0<y \leq 1/ (\ell \theta)  ~\text{and}~  x+y > 1+\varepsilon   \}, \\
H_3 &:= \{ (y,t,x) \in H  \colon 1/ (\ell \theta)<y<1/\theta  ~\text{and}~ 0<t<1-\varepsilon ~\text{and}~  x+y < 1+\varepsilon   \}, \\
H_4 &:= \{ (y,t,x) \in H  \colon 1/ (\ell \theta)<y<1/\theta  ~\text{and}~ 0<t<1-\varepsilon ~\text{and}~  x+y > 1+\varepsilon   \}, 
\end{align*}
Unfortunately, with varying $k$, $\varepsilon$, $\theta$ there is no uniform way to decompose $H_1, \dots ,H_4$ further into integrals with explicitely given limits. In the unconditional setting, namely with $\theta = 1/4$ fixed, every choice of parameters described in Table E provides less than 10 different integrals to calculate. For these choices we present close to optimal polynomials minimalizing the $\Omega_{k,\varepsilon}$ functional. 

\begin{center} 
\centering
\text{Table G. Upper bounds for $\Omega_{k,\varepsilon}$.}
\vspace{1mm}
\\
\renewcommand{\arraystretch}{1}
  \begin{tabular}{ | D || D|  I |  E | @{}m{0cm}@{}}\hline
     $k$ & $\varepsilon$ &  $f(1+\varepsilon-x)$ & bounds for $\Omega_k $  &  \rule{0pt}{3ex}  \\ \hline 
    $2$ & 1/3 & $ 1+ 5x + 3x^2$ & 4.6997 & \\
    $3$ & 1/4 & $ 1 + 7 x + 10x^2 $ &  7.7584   &  \\ 
    $4$ & 1/5 & $ 1 + 7 x + 19x^2 $ & 11.0533  & \\ 
    $5$ & 1/6 & $1 + 7 x + 33x^2$  & 14.5415  &  \\ 
    $6$ & 1/7 & $1+7x+51x^2$  & 18.1907  &  \\ 
    $7$ & 1/9 & $1+8x+70x^2$ &  21.9939  &  \\
    $8$ & 1/10 & $1+8x+102x^2 $ & 25.9038  &  \\ 
    $9$ & 1/10 & $1+5x+132x^2 $ & 29.9059   &  \\ 
    $10$ & 2/21 & $1+35x+ 30x^2 + 470x^3 $ &   33.9384        &  \\
      \hline
  \end{tabular}
\end{center}
These bounds are sufficient to prove the unconditional part of Theorem \ref{MAIN}.

\end{proof}


\bibliographystyle{amsplain}


\end{document}